\documentclass[11pt]{article}

\usepackage[utf8]{inputenc}
\usepackage[colorlinks=true
]{hyperref}
\hypersetup{urlcolor=blue, citecolor=red, linkcolor=blue}
\usepackage{mleftright}
\usepackage{stmaryrd}
\SetSymbolFont{stmry}{bold}{U}{stmry}{m}{n}
\usepackage{microtype}

\usepackage{bbm}
\usepackage{colortbl}

\usepackage{graphicx}
\usepackage[active]{srcltx}

\usepackage{amsmath,amssymb}

\usepackage{extarrows}

\numberwithin{equation}{section}

\usepackage{physics}
\usepackage{bm}
\usepackage{enumerate}
\usepackage{amsthm}
\usepackage[all,2cell]{xy}
\usepackage{mathrsfs}
\usepackage{mathtools}
\mathtoolsset{showonlyrefs}

\usepackage{tikz-cd}
\usetikzlibrary{cd}

\usepackage{xcolor}
\DeclareUnicodeCharacter{0301}{*************************************}

\hypersetup{urlcolor=blue, citecolor=red, linkcolor=blue}

\usepackage[a4paper, left=25mm, right=25mm, top=30mm, bottom=25mm]{geometry}

\usepackage{marginnote}

\usepackage{imakeidx}
\makeindex[intoc]

\usepackage[nottoc]{tocbibind}
\usepackage[colorinlistoftodos]{todonotes} 

\theoremstyle{plain}
\newtheorem{theorem}{Theorem}[section]
\newtheorem{proposition}[theorem]{Proposition}
\newtheorem{corollary}[theorem]{Corollary}
\newtheorem{lemma}[theorem]{Lemma}

\theoremstyle{definition}
\newtheorem{definition}[theorem]{Definition}
\newtheorem{remark}[theorem]{Remark}
\newtheorem{example}[theorem]{Example}

\AtBeginEnvironment{remark}{%
  \pushQED{\qed}%
}
\AtEndEnvironment{remark}{\popQED\endexample}

\AtBeginEnvironment{example}{%
  \pushQED{\qed}%
}
\AtEndEnvironment{example}{\popQED\endexample}

\newcommand\restr[2]{{
  \left.\kern-\nulldelimiterspace #1 \right|_{#2} 
}}

\newcommand{\R}{\mathbb{R}}

\renewcommand{\d}{\mathrm{d}}

\newcommand{\Cinfty}{\mathscr{C}^\infty}
\newcommand{\T}{\mathrm{T}}

\newcommand{\cT}{\mathrm{T}^\ast}

\newcommand*{\inn}[1]{\iota_{#1}}

\newcommand{\Lie}{\mathscr{L}}

\newcommand{\X}{\mathfrak{X}}

\newcommand{\dparder}[2]{\dfrac{\partial #1}{\partial #2}}

\DeclareMathOperator{\Ima}{Im}

\DeclareMathOperator{\supp}{Supp}

\DeclareMathAlphabet{\mathpzc}{OT1}{pzc}{m}{it}
\def\d{\mathrm{d}}

\newcommand{\dfn}[1]{\textbf{{#1}}}

\usepackage{xcolor}
\usepackage{fancyhdr}

\begin{document}


\vspace{3em}

{\huge\sffamily\raggedright Brackets in multicontact geometry\\[8px]and multisymplectization
}
\vspace{2em}

{\large\raggedright
    \today
}

\vspace{3em}

{\Large\raggedright\sffamily
    Manuel de León
}\vspace{1mm}\newline
{\raggedright
    Instituto de Ciencias Matemáticas (ICMAT--CSIC).\\
    C. Nicolás Cabrera, 13-15, Fuencarral, 28049 Madrid, Spain.
    \medskip
    
    Real Academia de Ciencias Exactas, Físicas y Naturales de España.\\
    C. Valverde, 22, 28004 Madrid, Spain.\\
    e-mail: \href{mailto:mdeleon@icmat.es}{mdeleon@icmat.es} --- orcid: \href{https://orcid.org/0000-0002-8028-2348}{0000-0002-8028-2348}
}

\medskip

{\Large\raggedright\sffamily
    Rubén Izquierdo-López
}\vspace{1mm}\newline
{\raggedright
    Instituto de Ciencias Matemáticas (ICMAT--CSIC).\\
    C. Nicolás Cabrera, 13-15, Fuencarral, 28049 Madrid, Spain.
    
    e-mail: \href{mailto:ruben.izquierdo@icmat.es}{ruben.izquierdo@icmat.es} --- orcid: \href{https://orcid.org/0009-0007-8747-344X}{0009-0007-8747-344X}
}

\medskip

{\Large\raggedright\sffamily
    Xavier Rivas
}\vspace{1mm}\newline
{\raggedright
    Department of Computer Engineering and Mathematics, Universitat Rovira i Virgili\\
    Avinguda Països Catalans 26, 43007 Tarragona, Spain\\
    e-mail: \href{mailto:xavier.rivas@urv.cat}{xavier.rivas@urv.cat} --- orcid: \href{https://orcid.org/0000-0002-4175-5157}{0000-0002-4175-5157}
}

\vspace{3em}

{\large\bf\raggedright
    Abstract
}\vspace{1mm}\newline
{\raggedright
In this paper we introduce a graded bracket of forms on multicontact manifolds. This bracket satisfies a graded Jacobi identity as well as two different versions of the Leibniz rule, one of them being a weak Leibniz rule, extending the well-known notions in contact geometry. In addition, we develop the multisymplectization of multicontact structures to relate these brackets to the ones present in multisymplectic geometry and obtain the field equations in an abstract context. The Jacobi bracket also permits to study the evolution of observables and study the dissipation phenomena, which we also address. Finally, we apply the results to classical dissipative field theories.
}

\vspace{3em}


\medskip


\medskip






\noindent {\bf Authors' contributions:} All authors contributed to the study conception and design. The manuscript was written and revised by all authors. All authors read and approved the final version.
\medskip

\noindent {\bf Competing Interests:} The authors have no competing interests to declare. 

\newpage

{\setcounter{tocdepth}{2}
\def\baselinestretch{1}
\small
\def\addvspace#1{\vskip 1pt}
\parskip 0pt plus 0.1mm
\tableofcontents
}

\pagestyle{fancy}

\fancyhead[L]{Brackets and multisymplectization}    
\fancyhead[C]{}                  
\fancyhead[R]{M. de León, R. Izquierdo-López and X. Rivas}       

\fancyfoot[L]{}     
\fancyfoot[C]{\thepage}                  
\fancyfoot[R]{}            

\setlength{\headheight}{17pt}

\renewcommand{\headrulewidth}{0.1pt}  
\renewcommand{\footrulewidth}{0pt}    

\renewcommand{\headrule}{%
    \vspace{3pt}                
    \hrule width\headwidth height 0.4pt 
    \vspace{0pt}                
}

\setlength{\headsep}{30pt}  

\section{Introduction}
The existence of brackets in classical field theory analogous to Poisson or Jacobi brackets in mechanics \cite{AMR_88,LR_85,LM_87,Vai_94} is of extreme importance. On the one hand, they are useful to show the evolution of any observable quantity and, on the other hand, they are the key to obtain the quantization of the given system, following Dirac's ideas. Indeed, the brackets bring the bridge between the geometry of phase space and the algebra of observables.

To our knowledge, the geometric version of Poisson brackets were first defined in \cite{CIL_96} for the multisymplectic formalism of classical field theories, just as a continuation of the study of abstract multisymplectic manifolds carried in \cite{CIL_99}. An important property of these brackets, already discussed in the work by Cantrijn, Ibort and de León, is that it induces a Lie algebra structure on $\Omega_H(M)/ \dd \Omega(M)$, the space of Hamiltonian forms modulo exact forms. These constructions have subsequently been used by other authors \cite{Bla_19,Bla_21,FG_13,gay-balmaz_new_2024}, and more recently developed within a more general and geometrical framework \cite{BMR_17,LI_25,Z_12} (see also a recent approach to classical field theories based on smooth sets of fields \cite{newone}).

On the other hand, dissipative field theories have recently been described by introducing the concept of multicontact structures \cite{LGMRR_23,LGMRR_25}, which generalize the usual contact ones. However, a bracket in multicontact field theories analogue to the Jacobi bracket in contact geometry is yet to be defined and studied. This is the first main goal of the present paper.

As we approached this problem, we discovered that a graded bracket can be associated to an arbitrary differential form on any smooth manifold, say $\Theta \in \Omega^n(M)$. The main ingredient is a special kind of multivector fields on the given manifold, which we call infinitesimal
conformal transformations of the given form $\Theta$. This definition mimics the situation on contact manifolds and their infinitesimal conformal transformations. {These multivector fields are then used to define the space of conformal Hamiltonian forms and the corresponding graded Jacobi bracket. This bracket is graded-skew-symmetric and satisfies the usual graded Jacobi identity (see Corollary \ref{cor:Jacobi_identity}). Furthermore, two different versions of the Leibniz rule appear: one corresponding to the Gerstenhaber algebra structure \cite{Ger_63} of infinitesimal conformal transformations (see Corollary \ref{corollary:Leibniz_rule}) and a second one generalizing the weak Leibniz rule present in contact geometry (see Proposition \ref{prop:Weak_Laibniz_Rule}).}

To study the relationship between these brackets and the graded Poisson brackets of multisymplectic geometry,
we extensively develop a pre-multisymplectization procedure extending the one already known for contact structures which works on any manifold equipped with an $n$-form $\Theta$. An interesting fact is that, as a consequence of our study, we can introduce a more general definition of multicontact form that the previously one defined in \cite{LGGMR_23,LGMRR_25}). Indeed, an $n$-form $\Theta$ is called multicontact if the following conditions hold:
$$
\ker_1 \Theta \cap \ker_1 \d \Theta = \{0\}\qquad\text{and}\qquad \ker_1\d \Theta \neq \{0\},
$$
where 
$$
\ker_1 \Theta = \{v \in \T M \colon \iota_v \Theta = 0\} \qquad \text{and} \qquad\ker_1 \d \Theta = \{v \in \T M \colon \iota_v \d \Theta = 0\}
\,,
$$
{being the conditions that guarantee that its canonical pre-multisymplectization (see Definition \ref{def:canonical_premultisymplectization}) is, in fact, a multisymplectization. It is worth mentioning that}
the second property $\ker_1\d \Theta \neq \{0\}$, besides avoiding the possibility that $\d\Theta =0$ (which would recover the case of a (pre)-multisymplectic form) is a technical condition that is used on several steps of our study.

{The second main goal of the present paper is to define dynamics, namely the action dependent Hamilton--De Donder--Weyl equations on an arbitrary multicontact manifold. Our approach is to identify the geometric objects that generalize the Jacobi bivector field and the Reeb vector field of contact manifolds.}  Indeed, we introduce the corresponding $\sharp$ mapping, which generalizes both the bivector field and the Reeb vector field.
This allows us to obtain a new expression for the bracket of forms on a multicontact manifold similar to the case of the Jacobi bracket in contact geometry. 

In order to introduce dynamics, { several steps need to be made. Firstly, we make a finer definition of the Reeb multivector field so that contracting with it is well defined for a wider family of forms. Secondly,} we define the so-called Hamiltonian subbundle of forms and identify then the good Hamiltonians, {which are the ones that allow for a definition of the dissipation form (see \cite{LGMRR_23}), which } 
has now a more clear and simple definition. Now, given such a good Hamiltonian, one can obtain the Hamilton--de Donder--Weyl equations; more insights are obtained when the multicontact form is of variational type, {where a tensor-like obstruction for every Hamiltonian to be a good Hamiltonian is identified (see Theorem \ref{thm:distortion})}. We also consider the evolution of forms and the notion of dissipated forms. The study ends with an application of the above theory to dissipative field theories.

In order to give a more complete view, we want to comment that the dissipative field equations can be described by other geometric structures, called $k$-contact \cite{LRS_24,GGMRR_20,GGMRR_21} and $k$-cocontact \cite{Riv_23} in the same way that the $k$-symplectic and $k$-cosymplectic geometries \cite{LSV_15} describe the classical field equations. Although the field equations are obtained with both formalisms, there is no known definition of brackets in these cases, which leaves multicontact structures to be more suitable for the study of action dependent field theories. Up to our knowledge, the only brackets defined in the $k$-contact setting was introduced in \cite{LRS_24} for vector-valued functions in order to study systems of ordinary differential equations.

The paper is structured as follows. Section \ref{section:Brackets_induced_by_form} is devoted to define brackets for arbitrary forms on a manifold. The multisymplectization method is developed in Section \ref{section:multisymplectization} to find the relations between these graded Jacobi brackets and graded Poisson brackets. In this section we also use the theory developed to motivate the definition of multicontact manifolds. We investigate the definition of the $\sharp$ mapping on arbitrary multicontact manifolds in Section \ref{section:sharp_mapping}, where we use it to give a different description of the brackets. These structures are then used in Section \ref{section:Field_equations} to define dynamics on general multicontact manifolds. Section \ref{section:Dissipative_field_theories} is devoted to apply the obtained results to the case of classical dissipative field theories. The paper finishes with some conclusions and further work.

\begin{center}{\bf Notation and conventions}
\end{center}
We will use the following notations and conventions throughout the paper; they are collected here to facilitate the reading.
\begin{itemize}\renewcommand\labelitemi{--}
    \item All manifolds are assumed $\Cinfty$-smooth, finite dimensional, Hausdorff and second countable.
    \item Einstein's summation convention is assumed throughout the text, unless stated otherwise.
    \item $\bigvee_p M$ denotes the vector bundle of $p$-vectors on $M$, $\bigwedge^p \T M$.
    \item $\mathfrak{X}^p(M) = \Gamma\left(\bigvee_p M\right)$ denotes the space of all multivector fields of order $p$.
    \item $\bigwedge^aM$ denotes the vector bundle of $a$-forms on $M$, $\bigwedge^a \T^\ast M$.
    \item $\Omega^a(M) = \Gamma\left (\bigwedge^a M\right )$ denotes the space of all $a$-forms on $M$.
    \item $[\cdot, \cdot]$ denotes the Schouten--Nijenhuis bracket on multivector fields. We use the sign conventions of \cite{Mar_97}.
    \item $\Lie_U \alpha = \d{\iota_U\alpha} - (-1)^p \iota \d{\alpha}$, for $U \in \mathfrak{X}^p(M)$,
    $\alpha \in \Omega^a(M)$ denotes the Lie derivative along multivector fields. 
    For a proof of its main properties, we refer to \cite{For_03}.
    \item For a subbundle $K \subseteq \bigvee_p M$, and for $a \geq p$, we denote by $$K^{\circ, a} = \{\alpha \in \bigwedge^a M\mid \iota_K \alpha = 0\}$$ the annihilator of order $a$ of $K$.
    \item Similarly, for a subbundle $S \subseteq \bigwedge^a M$, and $p \leq a$ we denote by 
    $$S^{\circ, p} = \{U \in \bigvee_p M\mid \iota_U S = 0\}$$ the annihilator of order $p$ of $S$.
    \item Given a form $\Theta \in \Omega^{n}(M)$, we denote $\ker_p \Theta = \{u \in \bigvee_p M \mid \iota_u \Theta = 0\}$.
    \item $\mathbb{R}_\times$ denotes $\mathbb{R}\setminus \{0\}$.
\end{itemize}

\section{Brackets induced by a differential \texorpdfstring{$n$}{}-form}
\label{section:Brackets_induced_by_form}

Let $M$ be a finite dimensional smooth manifold and $\Theta \in \Omega^n(M)$ be a smooth $n$-form on $M$. The following definition is an extension of the so-called conformal infinitesimal symmetries in contact geometry.

\begin{definition}
\label{def:conformal_symmetry}
For $p \leq n$, a multivector field $X \in \mathfrak{X}^p(M)$ is said to be an \dfn{infinitesimal conformal transformation} of $\Theta$ if there exists a multivector field $V \in \mathfrak{X}^{p-1}(M)$ such that
\[
\Lie_X \Theta = \iota_V \Theta\,,
\]
where $\Lie_X \Theta = \d \iota_X \Theta - (-1)^p \iota_X \d \Theta$ denotes the Lie derivative of $\Theta$ with respect to the multivector field $X$ (see \cite{For_03, Mar_97, Tul_74} for details).
\end{definition} 
\begin{remark}
  If $p = 1$ and $X \in \mathfrak{X}(M)$ is an infinitesimal conformal transformation, we have that there exists a certain $0$-multivector field, namely a function $g \in \Cinfty(M)$, such that
\[\Lie_X \Theta = g \cdot \Theta\,,\]
so that we recover the usual notion of an infinitesimal conformal symmetry.  
\end{remark}

\begin{definition}
\label{def:conformal_Hamiltonian_form}
A \dfn{conformal Hamiltonian form} is an $(n-p)$-form $\alpha \in \Omega^{n-p}(M)$ such that 
\[\iota_{X_\alpha} \Theta = -\alpha\,,\]
for some infinitesimal conformal transformation $X_\alpha \in \mathfrak{X}^p(M)$. The space of conformal Hamiltonian $a$-forms is denoted by $\Omega^a_H(M, \Theta)$. If the form to be used is clear form the context we shall simply use $\Omega^a_H(M)$. 
\end{definition}

Note that an $(n-p)$-form $\alpha$ is conformal Hamiltonian if and only if there exist multivector fields $X_\alpha \in \mathfrak{X}^p(M)$ and $V_\alpha \in \mathfrak{X} ^{p-1}(M)$ such that 
\begin{equation}
\label{eq:Hamiltonian_form}
    \inn{X_\alpha} \Theta = - \alpha\qquad\text{and}\qquad
    \inn{X_\alpha} \d \Theta = (-1)^{p+1}(\d \alpha + \iota_{V_\alpha}\Theta)\,.
\end{equation}
The following result will allow us to define a bracket on the space of conformal Hamiltonian forms.
\begin{theorem} 
\label{thm:Well_definedness}
Let $\alpha \in \Omega^{n-p}(M)$ and $\beta \in \Omega^{n-q}(M)$ be conformal Hamiltonian forms and  let $X_\alpha, V_\alpha, X_\beta$ and  $V_\beta$ be multivector fields satisfying
\begin{equation*}
    \iota_{X_\alpha} \Theta = - \alpha\,,\qquad
    \iota_{X_\alpha} \d \Theta = (-1)^{p+1}(\d \alpha + \iota_{V_\alpha}\Theta)\,,
\end{equation*}
and
\begin{equation*}
    \inn{X_\beta} \Theta = - \beta\,,\qquad
    \inn{X_\beta} \d \Theta = (-1)^{q+1}(\d \beta + \inn{V_\beta}\Theta)\,.
\end{equation*}
Then, the expression \[-\iota_{[X_\alpha, X_\beta]} \Theta\] only depends on $\alpha$ and $\beta$ and is also a conformal Hamiltonian form.
\end{theorem}
\begin{proof} Let us start by checking independence on the choice of multivector fields. We have (see \cite{For_03})
\begin{align*}
    -\iota_{[X_\alpha, X_\beta]} \Theta &= - (-1)^{(p-1)q} \Lie_{X_\alpha} \iota_{X_\beta} \Theta + \iota_{X_\beta}\Lie_{X_\alpha} \Theta \\
    &= - (-1)^{(p-1)q} \Lie_{X_\alpha} \iota_{X_\beta} \Theta+  \iota_{X_\beta} \iota_{V_\alpha} \Theta\\
    &= (-1)^{(p-1)q}(\iota_{V_\alpha} - \Lie_{X_\alpha}) \iota_{X_\beta} \Theta \\
    &=  (-1)^{(p-1)q}(\Lie_{X_\alpha}- \iota_{V_\alpha}) \beta\,,
\end{align*}
which implies that is independent on the choice of $X_\beta, V_\beta$. A symmetric argument proves the same result for $X_\alpha, V_\alpha$.

To see that it is again a conformal Hamiltonian form it is enough to show that $[X_\alpha, X_\beta]$ is again an infinitesimal conformal transformation for $\Theta$. Indeed,
\begin{align*}
    \Lie_{[X_\alpha, X_\beta]}\Theta &= (-1)^{(p-1)(q-1)} \Lie_{X_\alpha} \Lie_{X_\beta} \Theta - \Lie_{X_\beta} \Lie_{X_\alpha} \Theta\\
    &=  (-1)^{(p-1)(q-1)} \Lie_{X_\alpha} \iota_{V_\beta} \Theta - \Lie_{X_\beta} \iota_{V_\alpha} \Theta\\
    &= \iota_{[X_\alpha, V_\beta] } \Theta + \iota_{V_\beta} \Lie_{X_\alpha} \Theta - (-1)^{(p-1)(q-1)}(\iota_{[X_\beta, V_\alpha]} \Theta + \iota_{V_\alpha} \Lie_{X_\beta} \Theta)\\
    &= \iota_{[X_\alpha, V_\beta] } \Theta + \iota_{V_\beta} \iota_{V_\alpha} \Theta - (-1)^{(p-1)(q-1)}(\iota_{[X_\beta, V_\alpha]} \Theta + \iota_{V_\alpha} \iota_{V_\beta} \Theta)\\
    &= \iota_{[X_\alpha, V_\beta] } \Theta- (-1)^{(p-1)(q-1)}\iota_{[X_\beta, V_\alpha]} \Theta\\
    &= \iota_V \Theta\,,
\end{align*}
where 
\[V =[X_\alpha, V_\beta] - (-1)^{(p-1)(q-1)} [X_\beta, V_\alpha]\,,\]
concluding that it is a conformal symmetry.
\end{proof}

\begin{definition}
\label{def:Jacobi_bracket}
The \dfn{graded Jacobi bracket} of two conformal Hamiltonian forms $\alpha, \beta$ is
\[\{\alpha, \beta\} := - \iota_{[X_\alpha, X_\beta]} \Theta\,,\] where $X_\alpha, X_\beta$ are conformal symmetries for each of these forms. Well-definedness of this bracket follows from Theorem \ref{thm:Well_definedness}. For simplicity, we will refer to this bracket as the Jacobi bracket, provided there is no risk of confusion.
\end{definition}

The Jacobi bracket of conformal Hamiltonian forms defined by a differential $n$-form $\Theta \in \Omega^{n}(M)$ is a bilinear operation
\[\Omega^{a}_H(M) \otimes \Omega^b_H(M) \longrightarrow \Omega^{a + b - (n - 1)}_H(M)\,.\]

\begin{remark} \label{remark:Expression_for_Jacobi_bracket}
Notice that we may write 
$\{\alpha, \beta\} = (-1)^{(p-1)q} \left( \Lie_{X_\alpha}  - \iota_{V_\alpha}\right) \beta$. This formula resembles the Courant bracket used in the higher (or graded) generalization of Dirac geometry (see \cite{Z_12}). In fact, in Section \ref{section:sharp_mapping}, we will give a formula for the bracket in terms of the $\sharp$-morphism (see Theorem \ref{thm:bracket_formula_sharp}), which is a generalization of the formula found in the generalization of Dirac geometry (see \cite{LI_25} for the definition of the $\sharp$-morphisms and the expression of the Poisson brackets using them).
\end{remark}

\begin{remark} Suppose $\Theta = \eta$ a contact form on $M$. Then, the previous bracket generalizes the usual Jacobi bracket of contact geometry. Indeed, Eq. \eqref{eq:Hamiltonian_form} translates into
\[\iota_{X_f} \eta = -f\qquad\text{and}\qquad \iota_{X_f} \d \eta = \d f - R(f) \eta\,,\]
where $R$ denotes the \textit{Reeb vector field}, the unique vector field satisfying $\iota_{\mathcal{R}} \eta = 1$, and $\iota_{\mathcal{R}} \d \eta = 0$. Definition \ref{def:Jacobi_bracket} gives 
\[\{f, g\} = - \iota_{[X_f, X_g]} \eta\,,\]
which recovers the usual Jacobi bracket \cite{deLeon2019,LL_19}.
\end{remark}

The following result follows from the properties of the Schouten--Nijenhuis bracket.

\begin{corollary}
\label{cor:Jacobi_identity}
Let $\alpha\in \Omega^{n - p} (M)$, $\beta \in \Omega^{n - q}(M)$ and $\gamma \in \Omega^{n - r}(M)$ be conformal Hamiltonian forms. Then, the bracket of conformal Hamiltonian forms is graded skew-symmetric and satisfies the graded Jacobi identity, namely
\begin{enumerate}[{\rm(i)}]
    \item $\{\alpha, \beta\} = -(-1)^{(p-1)(q-1)}\{\beta, \alpha\}$,
    \item $(-1)^{(p-1)(r-1)} \{\alpha, \{\beta, \gamma\}\} + (-1)^{(r-1)(q-1)} \{\gamma, \{\alpha, \beta\}\} + (-1)^{(q-1)(p-1)} \{\beta, \{\gamma, \alpha\}\} = 0$.
\end{enumerate}
\end{corollary}

Note that infinitesimal conformal transformations are closed under exterior products, since
\begin{align}
    \Lie_{X\wedge Y}\Theta &= \Lie_Y\inn{X}\Theta + (-1)^q\inn{Y}\Lie_X\Theta \\
    &= (-1)^{p(q-1)}\left( \inn{[Y,X]}\Theta + \inn{X}\Lie_Y\Theta\right) + (-1)^q\inn{Y}\Lie_X\Theta\,, \label{eq:conformal_factor_wedge}
\end{align}
where $X\in\X^p(M)$ and $Y\in\X^q(M)$. This motivates the following definition.

\begin{definition}
    The \dfn{cup product} of conformal Hamiltonian forms is defined as \[\alpha  \vee \beta := - \iota_{X_\alpha \wedge X_\beta} \Theta\,.\]
\end{definition}

Note that the cup product of conformal Hamiltonian forms is well-defined since
$$ \alpha\vee\beta = \inn{X_\beta}\alpha = (-1)^{pq}\inn{X_\alpha}\beta\,. $$

\begin{remark} The cup product of two conformal Hamiltonian forms is again conformal Hamiltonian, since $X_\alpha \wedge X_\beta$ is again an infinitesimal conformal transformation.
\end{remark}

The cup product defines a bilinear operation
\[
\Omega_H^a(M) \otimes \Omega^b_H(M) \xlongrightarrow{\vee} \Omega^{a + b - n}_H(M)\,,
\]
which, using the properties of the Schouten--Nijenhuis bracket, is easily seen to satisfy the following property.

\begin{corollary}
\label{corollary:Leibniz_rule}
    Given $\alpha \in \Omega^{n - p}_H(M)$, $\beta \in \Omega^{n - q}_H(M)$ and $\gamma \in \Omega^{n - r}_H(M)$, the cup product satisfies the graded Leibniz rule
    \[\{\alpha, \beta \vee \gamma\} = \{\alpha, \beta\} \vee \gamma + (-1)^{(r-1)q} \beta \vee \{\alpha, \gamma\}\,.\]
\end{corollary}

Furthermore,

\begin{proposition}
\label{prop:Weak_Laibniz_Rule}
The Jacobi bracket satisfies the \textit{weak Leibniz rule}, namely
\[\supp \{\alpha,\beta\} \subseteq \supp \alpha \cap \supp \beta\,,\]
for every pair of Hamiltonian forms $\alpha \in \Omega^a _H(M, \Theta)$ and $\beta \in \Omega^b_H(M, \Theta)$, where $\supp\alpha$ denotes the support of $\alpha$.
\end{proposition}
\begin{proof} Indeed, this follows by noting that if $\alpha$ or $\beta$ vanish on an open subset, so does their bracket, since we can choose $X_\alpha$ or $X_\beta$ to be zero on this subset.
\end{proof}

This result, along with Corollary \ref{cor:Jacobi_identity} motivates the name graded Jacobi bracket given in Definition \ref{def:Jacobi_bracket}.

Observe that this endows the set of conformal Hamiltonian forms with a structure of a Gerstenhaber algebra \cite{Ger_63}, and the map $X\mapsto -\inn{X}\Theta$ is a Gerstenhaber algebra homomorphism between infinitesimal conformal transformations and conformal Hamiltonian forms.

\begin{remark}
    It is worth noting that, in the case where $\Theta = \eta$ defines a contact form on $M$, the cup product $f\vee g$ of two functions vanishes. Thus, we only have the weak Leibniz rule in the contact case.
\end{remark}

Let us turn into discussing an immediate relation between the brackets defined and the Poisson bracket defined by a multisymplectic form.

\begin{definition} A \textbf{pre-multisymplectic form} on $M$ is a closed $(n+1)$-form $\Omega \in \Omega^{n+1}(M)$. A pre-multisymplectic form is said to be \textbf{multisymplectic} or \textbf{non-degenerate} if the vector bundle map $\flat\colon \T M \rightarrow \bigwedge ^{n} M$ defined by $\flat(v) := \iota_v \Omega$ is a monomorphism.
\end{definition}

Every closed form induces a bracket in a suitable subspace of forms, which we define now. 

\begin{definition} Let $(M, \Omega)$ with $\Omega \in \Omega^{n+1}(M)$ be a (pre-)multisymplectic manifold. Then, a form $\alpha \in \Omega^{a}(M)$ is called \textbf{(pre-)multisymplectic Hamiltonian} if there exists a multivector field $X_\alpha \in \X^{n - a}(M)$ such that $\iota_{X_\alpha} \Omega = \d \alpha$. We denote the space of multisymplectic Hamiltonian $a$-forms by $\Omega^{a}_{H}(M, \Omega)$. If the form to be used is clear from the context we simply write $\Omega_H^a(M)$.
\end{definition}

\begin{remark} (Pre-)multisymplectic Hamiltonian forms are usually referred to as Hamiltonian. We prefer to incorporate this longer name in order to avoid confusion with conformal Hamiltonian forms.
\end{remark}

\begin{definition} Let $(M, \Omega)$, where $\Omega \in \Omega^{n+1}(M)$, be a pre-multisymplectic manifold and let $\alpha \in \Omega^{n-p}_H(M)$, $\beta \in \Omega^{n - q}_H(M)$ be pre-multisymplectic Hamiltonian forms. Their \textbf{Poisson bracket} is defined as \[\{\alpha, \beta\}_P = (-1)^{q-1}\iota_{X_\alpha \wedge X_\beta} \Omega\,.\]
\end{definition}

See \cite{CIL_96,CIL_99,RW_19} for further details on Poisson brackets in multisymplectic manifolds.

\begin{remark} The Jacobi bracket associated to a differential $n$-form generalizes the bracket of $1$-forms in symplectic geometry (see \cite{Vai_94}). Indeed, let $\Omega \in \Omega^{n+1}(M)$ be a multisymplectic form. Then, given multisymplectic Hamiltonian forms $\widetilde \alpha \in \Omega^{n-p}_{H}(M, \Omega)$ and $\widetilde \beta \in \Omega^{n-q}_{H}(M, \Omega)$, we have that their Poisson bracket is given by
\[\{\widetilde \alpha , \widetilde \beta\}_{P} = (-1)^{q-1} \iota_{X_{\widetilde \alpha} \wedge X_{\widetilde \beta}} \Omega\,,\]
where $X_{\widetilde \alpha} \in \X^p(M), X_{\widetilde \beta} \in \X^q(M)$ are multivector fields satisfying
\[
\iota_{X_{\widetilde \alpha}} \Omega = \d \widetilde \alpha\qquad\text{and}\qquad \iota_{X_{\widetilde\beta}} \Omega = \d \widetilde \beta\,.
\]
Then, both $X_{\widetilde \alpha}$ and $X_{\widetilde \beta}$ satisfy $\Lie_{X_{\widetilde \alpha}} \Omega = \Lie_{X_{\widetilde \beta}} \Omega= 0$ and, in particular, they define infinitesimal conformal transformations. Hence, $\d \widetilde \alpha$ and $\d \widetilde \beta$ may be interpreted as conformal Hamiltonian forms in the sense of Definition \ref{def:conformal_Hamiltonian_form} for the multivector fields $X_{\d \widetilde \alpha} = -X_{\widetilde \alpha}$ and $X_{\d \widetilde \beta} =- X_{\widetilde \beta}$, with conformal factors $V_{\d \widetilde \alpha} = 0$ and $V_{\d \widetilde \beta} = 0$, respectively. The Jacobi bracket of these two forms is related to the Poisson bracket of multisymplectic geometry as follows:
\begin{align*}
    \{\d \widetilde \alpha, \d \widetilde \beta\} &= - \inn{[X_{\widetilde \alpha}, X_{\widetilde \beta}]} \Omega = - (-1)^{(p-1)q} \Lie_{X_{\widetilde \alpha}} \inn{X_{\widetilde \beta}} \Omega + \inn{X_{\widetilde \beta}}\Lie_{X_{\widetilde \alpha}} \Omega \\
    &= - (-1)^{(p-1)q} \Lie_{X_{\widetilde \alpha}} \inn{X_{\widetilde \beta}} \Omega = - (-1)^{(p-1)q} \d \iota_{X_{\widetilde \alpha}} \iota_{X_{\widetilde \beta}} \Omega = (-1)^{q-1} \d \iota_{X_{\widetilde \alpha} \wedge X_{\widetilde \beta}} \Omega \\
    &= \d \{\widetilde \alpha, \widetilde \beta\}_P\,.
\end{align*}
As a particular case, for $n = 1$, we obtain $\{\d f, \d g\} = \d\{f, g\}_P$, where $\{f, g\}_P$ is the Poisson bracket of two functions from symplectic geometry (see \cite{AM_78,Arn_89}). Hence, we can think of the graded Jacobi bracket as a generalization of the graded Poisson bracket.
\end{remark}

\section{Multicontact structures and multisymplectization}
\label{section:multisymplectization}

In this section we study further relations between Jacobi and
Poisson brackets generalizing a procedure known in contact geometry: symplectization.

\subsection{General (pre-)multisymplectizations}
\begin{definition} An \textbf{homogeneous (pre-)multisymplectic} manifold is a pair $(\widetilde M, \Upsilon)$, where $\Upsilon \in \Omega^n(\widetilde M)$ is such that $\Omega = - \d \Upsilon$ defines a (pre-)multisymplectic form on $M$ and such that there exists a vector field $\Delta$, called \textbf{a Liouville vector field} satisfying $\iota_\Delta \Omega =  -\Upsilon$.
\end{definition}

\begin{remark} In the case of symplectic geometry, a homogeneous symplectic structure is simply an exact symplectic structure $\omega = -\d \upsilon$, since there always exists a vector field $\Delta$ such that $\iota_\Delta \omega = - \upsilon$. However, the same observation does not hold in multisymplectic geometry since, in general, we cannot guarantee the existence of a vector field $\Delta$ satisfying $\iota_{\Delta} \Omega = - \Upsilon$.
\end{remark}

Let $\Theta$ be an $n$-form on a smooth manifold $M$. 

\begin{definition}
\label{def:pre_multisymplectization}
Let $\tau\colon \widetilde M \rightarrow M$ be a (locally trivial) fiber bundle, where $\widetilde M$ is endowed with a homogeneous (pre-)multisymplectic structure $\Omega = - \d \Upsilon$. Then, it is said to be a \linebreak \textbf{(pre-)multisymplectization} of $(M, \Theta)$ if it admits a vertical Liouville vector field $\Delta$ with respect to the projection $\tau$, and $\Upsilon = \phi \cdot\tau^\ast \Theta$ for certain nowhere vanishing function $\phi \in \Cinfty(\widetilde M)$, called \textbf{the conformal factor}.
\end{definition}
{\begin{remark}
\label{remark:Delta_applied_to_phi}
Notice that from $\iota_\Delta \d (\phi\,\tau^\ast\Theta) = \phi\, \tau^\ast\Theta$ we must have $\Delta(\phi) \tau^\ast \Theta = \phi\, \tau^\ast \Theta$, so that if $\Theta$ is nowhere zero, we necessarily obtain $\Delta(\phi) = \phi$. A similar argument can be made for an $n$-form $\Theta$ that does not vanish on a dense subset of $M$.
\end{remark}
}

\begin{definition}
\label{def:canonical_premultisymplectization}
The \textbf{canonical (pre-)multisymplectization} of $(M, \Theta)$ is defined as $M \times \mathbb{R}_\times$, together with the homogeneous (pre-)multisymplectic structure $\Omega = - \d (z\,\Theta)$, where $z$ is the canonical coordinate in $\mathbb{R}_\times$. Notice that we can take as Liouville vector field $\Delta := z\pdv{z}$ and as conformal factor $\phi := z$.
\end{definition}

\begin{remark} A different choice of canonical multisymplectization could be $M \times \mathbb{R}$, with nowhere vanishing conformal factor $\phi(p, t) := e^t$. Both approaches are diffeomorphic (disregarding connected components), using the exponential mapping.
\end{remark}

\begin{proposition} 
\label{prop:multisymplectization}
Provided that $\ker_1 \d \Theta \neq \{0\}$, the canonical (pre-)multisymplectization of $(M, \Theta)$, namely $(M \times \mathbb{R}_\times, \Omega :=- \d(z \,\Theta))$, is a non-degenerate homogeneous multisymplectic manifold if and only if 
\[\ker_1 \Theta \cap \ker_1 \d \Theta = \{0\}\,.\]
\end{proposition}
\begin{proof}
Suppose first that $\ker_1 \Theta \cap \ker_1 \d \Theta = \{0\}$ and let $u \in \T (M \times \mathbb{R}_\times)$ such that $\inn{u} \Omega = 0$. Identifying $\T_{(p, a)} (M \times \mathbb{R}_\times) =\T_pM\times\T_a\R_\times\cong \T_p M \times \mathbb{R}$, we can write
\[u = v + \gamma \pdv{z}\,,\]
where $v \in \T_p M$ and $\gamma \in \mathbb{R}$. Then, $\iota_u \Omega = 0$ implies 
\[\d z \wedge \inn{v} \Theta - z \inn{v} \d \Theta - \gamma \Theta = 0\,.\]
We will prove that $v = 0$ and $\gamma = 0$. Contracting with $\pdv{z}$ yields $\iota_v \Theta = 0$. Now, since $\ker_1 \d \Theta \neq 0$, taking $w \in \ker_1 \d \Theta\setminus \{0\}$ and contracting by it we have $- \gamma \iota_w \Theta = 0$, which implies $\gamma = 0$, since $w \not \in \ker_1 \Theta$. Finally, we must have $z\, \iota_v \d \Theta = 0$ which implies $\iota_v \d \Theta = 0$, given that $z \neq 0$. We conclude that $v = 0$, since $v \in \ker _1 \Theta \cap \ker_1 \d \Theta = \{0\}$.

Conversely, assume that $- \d (z \, \Theta)$ defines a non-degenerate multisymplectic form. Consider $v \in \T M$ such that $\iota_v \Theta = \iota_v \d \Theta = 0$. Then, identifying $\T (M \times \mathbb{R}_\times) = \T M \times \T \mathbb{R}_\times$ we have $\iota_u \d (z\,\Theta) = 0$, which implies $u = 0$, as we wanted to see.
\end{proof}



Let us turn back to the theory of arbitrary (pre-)multisymplectizations. In particular, we will prove that every infinitesimal conformal transformation on the base can be lifted to a homogenous multivector field on the (pre-)multisymplectization.

\begin{theorem}
\label{thm:multivector_field_lift}
Let {$\Theta$ be a nowhere vanishing $n$-form on $M$ and} let
$\tau\colon \widetilde M \rightarrow M$, with $(\widetilde M, - \d \Upsilon)$ a (pre-)multisymplectization of $(M, \Theta)$ with conformal factor $\phi \in \Cinfty(\widetilde M)$. Then, for every infinitesimal conformal transformation $X \in \X^p(M)$ of $\Theta$ there exists a unique-up-to-$\ker_p \d\Upsilon$ multivector field $\widetilde X \in \X^p(\widetilde M)$ such that
\[
\Lie_{\widetilde X} \Upsilon = 0 \qquad\text{and}\qquad \tau_\ast \widetilde X = X\,.
\]
\end{theorem}
To prove this result, we will need the existence of a particular Ehresmann connection on $\tau\colon \widetilde M \rightarrow M$, namely a subbundle $H \subseteq \T \widetilde M$ such that $\ker \d \tau \oplus H =\T \widetilde M$ and $H \subseteq \ker \d \phi$. The following lemma ensures the existence of such a connection.
\begin{lemma}
\label{lemma:connection}
There exists an Ehresmann connection $H$ on $\tau\colon \widetilde M \rightarrow M$ such that $H \subseteq \ker \d \phi$.
\end{lemma}
\begin{proof}
Notice that we have $\ker \d \tau +\ker\d \phi = \T \widetilde M$. Indeed, {following Remark \ref{remark:Delta_applied_to_phi}, we have} $\d \phi(\Delta) = \phi$, which implies $\d \phi \neq 0$ and $\Delta \not \in \ker \d \phi$. Hence,
\[\rank \ker \d_p\phi = \dim \widetilde M - 1\,.\]
Now, since $\Delta_p \notin \ker \d_p \phi$, by a dimension comparison and using that $\Delta$ is $\tau$-vertical, we must have
\[\T \widetilde M = \ker \d\phi \oplus \langle \Delta\rangle \subseteq \ker \d \phi + \ker \d\tau\,.\]
Furthermore, since $\ker \d \phi + \ker \d \tau = \T \widetilde M$, we have that $\ker \d \phi \cap \ker \d \tau$ has constant rank and, therefore, $\ker \d \phi / (\ker \d \phi \cap \ker \d \tau) \rightarrow M$ is a well-defined vector bundle. Notice that this induces a short exact sequence
\[0 \longrightarrow \ker \d \phi \cap \ker \d \tau \longrightarrow \ker \d \phi \longrightarrow \ker \d \phi /(\ker \d \phi \cap \ker \d \tau) \longrightarrow 0\,.\] 
Taking an arbitrary Euclidean structure on $\ker \d \phi \rightarrow \widetilde M$, we may define $H$ through a splitting of the previous short exact sequence. It suffices to take $H$ as the orthogonal complement of $\ker \d \phi \cap \ker \d \tau$ with respect to the chosen inner product. Indeed, by construction, it satisfies $H \cap \ker \d \tau = \{0\}$ and, furthermore, \[\rank H = \rank \ker \d \phi - \rank (\ker \d \phi  \cap \ker \d \tau) = \rank \T \widetilde M - \rank  \ker \d \tau,\]
so that $H \oplus \ker \d \tau = \T \widetilde M$, defining the desired Ehresmann connection and finishing the proof.
\end{proof}
An Ehresmann connection $H$ on $\tau\colon\widetilde M \rightarrow M$ induces a  map $\X^p(M) \to \X^p(\widetilde M)$ given by
\[U = X_1 \wedge \cdots \wedge X_p \longmapsto U^h := X_1^h \wedge \cdots \wedge X_p^h\,,\]
where $X_i^h$ is the unique vector field on $\widetilde M$ satisfying 
\[\tau_\ast X_i^h = X_i\qquad\text{and}\qquad X_i^h \in H\,.\]
\begin{lemma}
\label{lemma:contraction}
Let $H$ be the Ehresmann connection from Lemma \ref{lemma:connection}. Then, for every $X \in \X^p(M)$ and $\alpha \in \Omega^a(M)$, we have
\[
\iota_{X^h} (\d \phi \wedge \tau^\ast \alpha) = (-1)^p \d \phi \wedge \tau^\ast(\iota_X \alpha)\,.
\]
\end{lemma}
\begin{proof} Suppose $X = X_1 \wedge \cdots \wedge X_p$, where $X_i \in \X(M)$. Then, $X^h = X_1^h \wedge \cdots \wedge X^h_p$. Notice that
\begin{align*}
    \iota_{X^h} (\d \phi \wedge \tau^\ast \alpha) &= \iota_{X_1^h \wedge \cdots \wedge X^h_p}(\d \phi \wedge \tau^\ast \alpha) 
    = \iota_{X_2^h \wedge \cdots \wedge X^h_p} \iota_{X_1^h}(\d \phi \wedge \tau^\ast \alpha)\\
    & =  \iota_{X_2^h \wedge \cdots \wedge X^h_p}(\d \phi(X_1^h) \alpha - \d \phi \wedge \tau^\ast \iota_{X_1} \alpha) = -\iota_{X_2^h \wedge \cdots \wedge X^h_p}( \d \phi \wedge \tau^\ast \iota_{X_1} \alpha)\,,
\end{align*}
where in the last equality we have used that $X_1^h \in \ker \d \phi$. Iterating the last argument we conclude that $\iota_{X^h} (\d \phi \wedge \tau^\ast \alpha) = (-1)^p \d \phi \wedge \tau^\ast(\iota_X \alpha)$.
\end{proof}

\begin{proof}[Proof of Theorem \ref{thm:multivector_field_lift}] Let $H$ be the Ehresmann connection from Lemma \ref{lemma:connection} and let $X \in \X^p(M)$ be an infinitesimal conformal transformation of $\Theta$. Define \[\widetilde X := X^h + \Delta \wedge Y^h\,,\] where $Y \in \X^{p-1}(M)$ is a multivector field which is going to be specified later. Notice that $\tau_\ast \widetilde X = X$. We will look for $Y$ such that $0 = \Lie_{\widetilde X} \Upsilon = -\Lie_{\widetilde X} (\phi \cdot \tau^\ast \Theta)$. Using Lemma \ref{lemma:contraction}, we have
\begin{align*}
\Lie_{X^h + \Delta \wedge Y^h} (\phi\cdot  \tau^\ast\Theta) &= \d \iota_{X^h + \Delta \wedge Y^h} (\phi \cdot \tau^\ast \Theta) - (-1)^p \iota_{X^h + \Delta \wedge Y^h} \d (\phi \cdot \tau^\ast \Theta)\\
&= \d (\phi \tau^\ast (\iota_X \Theta)) - (-1)^p\iota_{X^h + \Delta \wedge Y^h} (\d \phi \wedge \tau^\ast \Theta + \phi \tau^\ast \d \Theta) \\
&= \d \phi \wedge \tau^\ast(\iota_X \Theta) + \phi \d(\tau^\ast \iota_X \Theta) - \d \phi \wedge \tau^\ast( \iota_X \Theta)\\
&\quad- (-1)^p\d \phi(\Delta) \tau^\ast ( \iota_Y \Theta)- (-1)^p \phi \tau^\ast (\iota_X \d \Theta)\\
&= \phi \tau^\ast \Lie_X \Theta - (-1)^p\d \phi(\Delta) \tau^\ast ( \iota_Y \Theta) \\
&= \phi \tau^\ast \left( \Lie_X \Theta - (-1)^p  \iota_Y \Theta\right)\,,
\end{align*}
since $\d \phi(\Delta) = \phi$ ({see Remark \ref{remark:Delta_applied_to_phi}}). Recall that $X$ is an infinitesimal conformal transformation and, hence, there exists a multivector field $V \in \X^{p-1}(M)$ such that $\Lie_X \Theta = \iota_V \Theta$. Therefore, it is enough to take $Y := (-1)^p V$ to conclude existence.

Let us show uniqueness-up-to-$\ker_p \d\Upsilon$. Consider multivector fields $\widetilde X_1, \widetilde X_2$ such that $\Lie_{\widetilde X_i} \Upsilon = 0$ and $\tau_\ast \widetilde X_i = X$. Then,
\begin{align*}
    0 &= \Lie_{\widetilde X_i} \Upsilon = - (-1)^p \iota_{\widetilde X_i} \d \Upsilon + \d \iota_{\widetilde X_i} \Upsilon \\
    &= - (-1)^p \iota_{\widetilde X_i} \d \Upsilon+ \d \iota_{\widetilde X_i} (\phi \cdot \tau^\ast \Theta)\\
    &= - (-1)^p \iota_{\widetilde X_i} \d \Upsilon+ \d (\phi \cdot \tau^\ast \iota_X \Theta)\,.
\end{align*}
Hence, \[0 = \iota_{\widetilde X_1} \d \Upsilon - \iota_{\widetilde X_2} \d \Upsilon = \iota_{\widetilde X_1 - \widetilde X_2} \d \Upsilon\,,\]
which finishes the proof.
\end{proof}

Furthermore, we have that the lift of infinitesimal conformal transformations translates into a bracket preserving lift of conformal Hamiltonian forms to (pre-)multisymplectic Hamiltonian forms.
\begin{theorem}
\label{thm:multisymplectization_and_bracket}
Let {$\Theta$ be a nowhere vanishing $n$-form on $M$ and} let $\tau\colon \widetilde M \rightarrow M$, with $(\widetilde M, - \d \Upsilon)$, be a (pre-)multisymplectization of $(M, \Theta)$ with conformal factor $\phi \in \Cinfty(\widetilde M)$. Then, the map 
\[
\begin{array}{rccc}
\Psi\colon & \Omega_H^{n - p}(M, \Theta) & \longrightarrow & \Omega_H^{n- p}(\widetilde M, \Omega  = - \d \Upsilon)\\
& \alpha & \longmapsto & (-1)^{p+1}\phi \cdot \tau^\ast \alpha
\end{array}
\]
is well defined and, for every $\alpha \in \Omega^{n - p}_H(M, \Theta)$ and $\beta \in \Omega^{n - q}_H(M, \Theta)$ it satisfies \[\{\Psi(\alpha), \Psi(\beta)\}_P = \Psi(\{\alpha, \beta\}) {+(-1)^q \d \Psi(\alpha \vee \beta)}\,.\]
\end{theorem}

\begin{proof}
     Let $\alpha \in \Omega^{n-p}_H(M, \Theta)$ and let $X_\alpha \in \X^p(M), V_\alpha \in X^{p-1}(M)$ be multivector fields such that
     \[\alpha = -\iota_{X_\alpha} \Theta\qquad\text{and}\qquad \Lie_{X_\alpha}\Theta = \iota_{V_\alpha} \Theta\,.\]
     Take $\widetilde X \in \X^p(M)$ to be the multivector field from Theorem \ref{thm:multivector_field_lift}. Then, we have
     \begin{align*}
         \iota_{\widetilde X} (- \d \Upsilon)& = - \iota_{\widetilde X} \d \Upsilon = - (-1)^p \d \iota_{\widetilde X} \Upsilon + (-1)^p \Lie_{\widetilde X} \Upsilon\\
         &= - (-1)^p \d \iota_{\widetilde X} (\phi \cdot \tau^\ast \Theta) = (-1)^{p+1}\d(\phi\cdot \tau^\ast \alpha)\,,
     \end{align*}
     which proves that $\phi \cdot \d \alpha$ is Hamiltonian and hence $\psi$ is well defined.

     Consider $\alpha \in \Omega^{n-p}_H(M, \Theta)$ and $\beta \in \Omega^{n-q}_H(M, \Theta)$. Let $X_\alpha, V_\alpha, X_\beta, V_\beta$ be multivector fields such that
     \begin{equation*}
        \alpha = - \iota_{X_\alpha} \Theta\,, \qquad \Lie_{X_\alpha} \Theta = \iota_{V_\alpha} \Theta\,,\qquad
        \beta = - \iota_{X_\beta} \Theta\,, \qquad \Lie_{X_\beta} \Theta = \iota_{V_\beta} \Theta\,.
     \end{equation*}
     Defining $H$ to be the Ehresmann connection from Lemma \ref{lemma:connection}, following Theorem \ref{thm:multivector_field_lift}, we can take as lift of $X_\alpha$ and $X_\beta$
     \[
     \widetilde X_\alpha = X_\alpha^h + (-1)^p \Delta \wedge V_\alpha\qquad\text{and}\qquad \widetilde X_\beta = X_\beta^h + (-1)^q \Delta \wedge V_\beta\,,
     \]
     respectively. As we know from the previous calculation, these lifts are the vector fields corresponding to the (pre-)multisymplectic Hamiltonian forms $\Psi(\alpha) = (-1)^{p+1} \phi\, \tau^\ast \alpha$ and $\Psi(\beta) =(-1)^{q+1} \phi\, \tau^\ast \beta$. Therefore, using that $\Lie_{\widetilde X_\alpha \wedge \widetilde X_\beta} \Upsilon = (-1)^q\iota_{\widetilde X_\beta} \Lie_{\widetilde X_\alpha} \Upsilon + \Lie_{\widetilde X_\beta} \iota_{\widetilde X_\alpha} \Upsilon$ (see \cite{For_03}), we have
     \begin{align*}
         \{\Psi(\alpha), \Psi(\beta)\}_P &= (-1)^{q-1} \iota_{X_{\psi(\alpha)} \wedge X_{\psi(\beta)}} \Omega = (-1)^q \iota_{\widetilde X_\alpha \wedge \widetilde X_\beta} \d \Upsilon\\
         &= (-1)^p\left( -\Lie_{\widetilde X_\alpha \wedge \widetilde X_\beta} \Upsilon + \d \iota_{\widetilde X_\alpha \wedge \widetilde X_\beta} \Upsilon\right) \\
         &= (-1)^p \left( - (-1)^q\iota_{\widetilde X_\beta} \Lie_{\widetilde X_\alpha} \Upsilon - \Lie_{\widetilde X_\beta} \iota_{\widetilde X_\alpha} \Upsilon+ \d (\phi\, \tau^\ast \iota_{X_\alpha \wedge X_\beta} \Theta)\right)\\
         &=(-1)^p\left( - \Lie_{\widetilde X_\beta} \iota_{\widetilde X_\alpha} \Upsilon - \d \left(\phi \, \tau^\ast (\alpha \vee \beta)\right)\right),
        \end{align*}
    where in the last equality we have used that $\Lie_{\widetilde X_{\alpha}} \Upsilon =0$, being the definition of the lift from Theorem \ref{thm:multivector_field_lift}. Now, using that
    \begin{align*}
         \Lie_{\widetilde X_\beta} \iota_{\widetilde X_\alpha} \Upsilon &= (-1)^{(q-1)p} \iota_{[\widetilde X_{\beta}, \widetilde X_{\alpha}]} \Upsilon + (-1)^{(q-1)p} \iota_{\widetilde X_\alpha} \Lie_{\widetilde X_{\beta}} \Upsilon\\
         &= (-1)^{(q-1)p} \iota_{[\widetilde X_{\beta}, \widetilde X_{\alpha}]} \Upsilon = (-1)^{(q-1)p} \phi \, \tau^\ast \iota_{[X_\beta, X_\alpha]} \Theta\\
         &= - (-1)^{(q-1)p} \phi \, \tau^\ast \{\beta, \alpha\}\,,
    \end{align*}
    we get
        \begin{align*}
         \{\Psi(\alpha), \Psi(\beta)\}_P
         &= (-1)^p\left( (-1)^{(q-1)p}\phi\, \tau^\ast \{\beta, \alpha\} - \d \left(\phi \, \tau^\ast (\alpha \vee \beta)\right)\right)\\
         &= (-1)^{p+q}\phi \, \tau^\ast\{\alpha, \beta\}+ (-1)^{p+1}\d \left(\phi \, \tau^\ast (\alpha \vee \beta)\right)\\
         &= \Psi(\{\alpha, \beta\})+ (-1)^q \d \Psi(\alpha \vee \beta),
     \end{align*}
     which finishes the proof.
\end{proof}

\begin{remark} In the case where $\Theta = \eta$ is a contact form, Theorem \ref{thm:multisymplectization_and_bracket} translates into a map
$$
\begin{array}{rccc}
\Psi\colon & \Cinfty(M) & \longrightarrow & \Cinfty(\widetilde M) \\
& f & \longmapsto & \Psi(f) = \phi\cdot f\,,
\end{array}
$$
that satisfies \[\{\Psi(f), \Psi(g)\}_P = \Psi(\{f, g\})\,,\]
since $f \vee g = - \iota_{X_f \wedge X_g} \eta = 0$, so that we recover the usual case of symplectizations of contact manifolds.
\end{remark}

More generally, 
\begin{proposition} Let {$\Theta$ be a nowhere vanishing $n$-form on $M$ and} consider $\alpha \in \Omega^a_H(M, \Theta)$ and $\beta \in \Omega^b_H(M, \Theta)$ with $a+ b < n$. Then, the map $\Psi$ defined in Theorem \ref{thm:multisymplectization_and_bracket} satisfies 
\[\{\Psi(\alpha), \Psi(\beta)\}_P = \Psi(\{\alpha, \beta\}).\]
\end{proposition}
\begin{proof} Notice that we have $\alpha \vee \beta = - \iota_{X_\alpha \wedge X_\beta} \Theta$, and $\deg (X_\alpha \wedge X_\beta) = (n - a) + (n - b) = 2n - (a+b) > n$, which implies $\alpha \vee \beta = 0$, and finishes the proof using the formula in Theorem \ref{thm:multisymplectization_and_bracket}. 
\end{proof}

In the case of the canonical (pre-)multisymplectization, the previous results translate into the following.

\begin{corollary} For every conformal transformation of {a nowhere vanishing form} $\Theta$, $X \in \mathfrak{X}^p(M)$, there exists a unique-up-to-$\ker_p  \d (z \, \Theta)$ multivector field $\widetilde X \in \mathfrak{X}^p(M \times \mathbb{R}_\times)$ such that \[\tau_\ast \widetilde X = X\qquad\text{and}\qquad \Lie_{\widetilde X} (z \,\Theta ) = 0\,.\] 
\end{corollary}

\begin{corollary} For {a nowhere vanishing form,} $\Theta$, the map
$$
\begin{array}{rccc}
    \Psi\colon & \Omega^{n - p}_H(M) & \longrightarrow & \Omega^{n - p}_H(M \times \mathbb{R}_\times) \\
     & \alpha & \longmapsto & \Psi(\alpha) = (-1)^{p+1} z \,\tau^\ast \alpha
\end{array}
$$
satisfies
\[\{\Psi(\alpha),  \Psi(\beta)\}_P = \Psi({\{\alpha, \beta\}}) +(-1)^q \d \Psi(\alpha \vee \beta)\,,\]
for every $\alpha \in \Omega^{n-p}_H(M, \Theta)$ and $\beta \in \Omega^{n-q}_H(M, \Theta)$.
\end{corollary}

Theorem \ref{thm:multisymplectization_and_bracket} may be expanded upon to give a formula relating {\it higher Jacobi identities}. Before proving this, let us first give a couple of formulas that we will find useful:

\begin{lemma}
\label{lemma:conformal_factor_of_wedge} Let $\Theta$ be a nowhere vanishing $n$-form on a smooth manifold $M$. Let $X_1, \dotsc, X_m$ be infinitesimal conformal transformations of $\Theta$ with conformal factor $g_1, \dotsc, g_m \in \Cinfty(M)$, respectively. Then, we may take as conformal factor of $X_1\wedge \dotsb \wedge X_m$ the following multivector field:
\begin{align*}
    \sum_{1 \leq i < j \leq m} (-1)^{m + i  + j} [X_i, X_j] \wedge X_1  \wedge \cdots \wedge\widehat{X_i}\wedge \cdots \wedge\widehat{X_j}\wedge \cdots \wedge X_m\\
    + \sum_{1 \leq i \leq m} (-1)^{m +i} g_i X_1 \wedge \cdots \wedge \widehat{X}_i \wedge \cdots \wedge X_m\,.
\end{align*}
\end{lemma}

\begin{proof} As we know, the wedge product of infinitesimal conformal transformations is conformal again, by Eq. \eqref{eq:conformal_factor_wedge}. Now the formula follows from an inductive reasoning. Indeed, the case $m = 2$ is clear from Eq. \eqref{eq:conformal_factor_wedge}. Now, for the inductive step, let us assume that the conformal factor of $X_1 \wedge\cdots \wedge X_{m-1}$ is 
\begin{align*}
     U = &\sum_{1 \leq i < j \leq m-1} (-1)^{m -1 + i  + j} [X_i, X_j] \wedge X_1  \wedge \cdots \wedge\widehat{X_i}\wedge \cdots \wedge\widehat{X_j}\wedge \cdots \wedge X_{m-1}\\
    &+ \sum_{1 \leq i \leq m-1 } (-1)^{m-1+i} g_i X_1 \wedge \cdots \wedge \widehat{X}_i \wedge \cdots \wedge X_{m-1}\,.
\end{align*}
Then, by applying again Eq. \eqref{eq:conformal_factor_wedge}, we have that, denoting the conformal factor of $X_1 \wedge \cdots \wedge X_m$ by $V$: 
\begin{align*}
    V &= [X_m, X_1 \wedge \cdots \wedge X_{m-1}] + g_m X_1 \wedge \cdots \wedge X_{m-1} - U \wedge X_m \\
    &= [X_m, X_1 \wedge \cdots \wedge X_{m-1}] + g_m X_1 \wedge \cdots \wedge X_{m-1}\\
    &\quad- \sum_{1 \leq i < j \leq m-1} (-1)^{m -1 + i  + j} [X_i, X_j] \wedge X_1  \wedge \cdots \wedge\widehat{X_i}\wedge \cdots \wedge\widehat{X_j}\wedge \cdots \wedge X_{m}\\
    &\quad- \sum_{1 \leq i \leq m-1 } (-1)^{m-1+i} g_i X_1 \wedge \cdots \wedge \widehat{X}_i \wedge \cdots \wedge X_{m}\\
    &= [X_m, X_1 \wedge \cdots \wedge X_{m-1}] + g_m X_1 \wedge \cdots \wedge X_{m-1}\\
    &\quad +\sum_{1 \leq i < j \leq m-1} (-1)^{m + i  + j} [X_i, X_j] \wedge X_1  \wedge \cdots \wedge\widehat{X_i}\wedge \cdots \wedge\widehat{X_j}\wedge \cdots \wedge X_{m}\\
    &\quad+\sum_{1 \leq i \leq m-1 } (-1)^{m+i} g_i X_1 \wedge \cdots \wedge \widehat{X}_i \wedge \cdots \wedge X_{m}\,,
\end{align*}
which is clearly equal to the desired expression.
\end{proof}

\begin{lemma}
\label{lemma:Lift_of_wedge}
Let $X_1, \dotsc, X_m$ be infinitesimal conformal transformations of $(M, \Theta)$ and let $\tau \colon \widetilde M \rightarrow M$, with $(\widetilde M, -\d \Upsilon)$ a (pre-)multisymplectic manifold be a multisymplectization with conformal factor $\phi$. Let $H$ be the Ehresmann connection of Lemma \ref{lemma:connection} and denote by $X^h$ its horizontal lift. Denote by $\widetilde X_i$ a lift of each vector field given by Theorem \ref{thm:multivector_field_lift} and denote by $U$ a lift of $X_1 \wedge \cdots \wedge X_m$. Then, we have 
\begin{align*}
    \widetilde X_1 \wedge \cdots \wedge \widetilde X_m &= U + \Delta \wedge \sum_{1 \leq i < j \leq m} (-1)^{i+j - 1}\left([X_i, X_j]\wedge X_1  \wedge \cdots \wedge\widehat{X_i}\wedge \cdots \wedge\widehat{X_j}\wedge \cdots \wedge X_{m} \right)^h\\
     &\operatorname{mod}\ker_p \d \Upsilon\,.
\end{align*}
\end{lemma}
\begin{proof} Denote by $g_i$ the conformal factor of each $X_i$. In light of the proof of Theorem \ref{thm:multivector_field_lift}, we have
\[
\widetilde X_i := X_i^h - g_i \Delta \qquad \operatorname{mod} \ker_1 \d \Upsilon\,.
\]
Then,
\begin{align*}
    \widetilde X_1 \wedge \cdots \wedge \widetilde X_m &= (X_1 \wedge\cdots \wedge X_m)^h + \Delta \wedge \sum_{1 \leq i \leq m} (-1)^{i} g_i \left( X_1 \wedge \cdots \wedge \widehat{X_i} \wedge \cdots \wedge X_m\right)^h\\
    & \operatorname{ mod} \ker_p \d \Upsilon\,.
\end{align*}
Let $V$ denote the conformal factor of $X_1 \wedge \cdots \wedge X_m$ given by Lemma \ref{lemma:conformal_factor_of_wedge}. Then, again using the same formula obtained in the proof of Theorem \ref{thm:multivector_field_lift}, we have 
\[
U = (X_1 \wedge \cdots \wedge X_m)^h + (-1)^m \Delta \wedge V^h \, \operatorname{mod} \ker_p \d \Upsilon\,,
\]
where 
\begin{align*}
    V =   \sum_{1 \leq i < j \leq m} (-1)^{m + i  + j} [X_i, X_j] \wedge X_1  \wedge \cdots \wedge\widehat{X_i}\wedge \cdots \wedge\widehat{X_j}\wedge \cdots \wedge X_m\\
    + \sum_{1 \leq i \leq m} (-1)^{m +i} g_i X_1 \wedge \cdots \wedge \widehat{X}_i \wedge \cdots \wedge X_m\,,
\end{align*}
from which the result follows.
\end{proof}

We are now ready to prove the refinement of Theorem \ref{thm:multisymplectization_and_bracket}. 

\begin{theorem}
\label{thm:L_infinity}
Let $\Theta $ be a nowhere vanishing $n$-form on $M$ and let $\tau \colon \widetilde M \rightarrow M$ be a (pre-)multisymplectization, with $(\widetilde M, - \d \Upsilon)$ a (pre-)multisymplectic manifold, with nowhere vanishing conformal factor $\phi \in \Cinfty(\widetilde M)$. Denote by $\Psi_1 = \Psi$ the map from Theorem \ref{thm:multisymplectization_and_bracket}. For $m \geq 2$, define
\[
\Psi_m  \colon \Omega_H^{n-1}(M, \Theta)^{\otimes m} \longrightarrow \Omega^{n+1-m}(\widetilde M)
\]
as $\Psi_m(\alpha_1, \dotsc, \alpha_m) :=  \phi \cdot  \tau^\ast \left( \alpha_1 \vee \cdots \vee \alpha _m\right)$ and \[
\ell_m \colon \Omega^{n-1}_H(\widetilde M, - \d \Upsilon)^{\otimes m} \longrightarrow \Omega^{n+1-m}(\widetilde M)
\]
by $\ell_m(\beta_1, \dotsc,\beta_m) := -\iota_{X_{\beta_1} \wedge \dotsb \wedge X_{\beta_m}} \d \Upsilon$. Then, for every $\alpha_1, \dotsc, \alpha_m \in \Omega^{n-1}_H(M, \Theta)$, we have
\begin{align*}
    \ell_m (\Psi(\alpha_1), \dotsc, \Psi(\alpha_m)) &= (-1)^{m+1}\d \Psi_{m} ( \alpha_1, \dotsc, \alpha_m) \\
    &\quad +\sum_{1 \leq i < j \leq m} (-1)^{i+j +1} \Psi_{m-1}(\{\alpha_i, \alpha_j\}, \alpha_1, \dotsc, \widehat \alpha_i, \dotsc, \widehat \alpha_j, \dotsc, \alpha_m) 
\end{align*}
\end{theorem}
\begin{proof} Indeed, by definition, denoting by $X_i$ the infinitesimal conformal transformation associated to $\alpha_i$, and by $\widetilde X_i$ their respective lifts, we have
\begin{align*}
    \ell_m(\Psi(\alpha_1), \dotsc, \Psi(\alpha_m)) = - \iota_{\widetilde X_1 \wedge \cdots \wedge \widetilde X_m} \d \Upsilon\,,
\end{align*}
which by Lemma \ref{lemma:Lift_of_wedge}, we may write as 
\begin{align*}
    \ell_m(\Psi(\alpha_1), \dotsc, \Psi(\alpha_m)) &= - \iota_U \d \Upsilon\\
    &\quad  + \sum_{1 \leq i < j \leq m}(-1)^{i+j}  \iota_{\left(\Delta \wedge \left([X_i, X_j]\wedge X_1  \wedge \cdots \wedge\widehat{X_i}\wedge \cdots \wedge\widehat{X_j}\wedge \cdots \wedge X_{m} \right)^h\right)} \d \Upsilon\,,
\end{align*}
where $U$ denotes a lift of $X_1 \wedge \cdots \wedge X_m$. Since $\Lie_U \Upsilon = 0$, we have 
\begin{align*}
    \ell_m(\Psi(\alpha_1), \dotsc, \Psi(\alpha_m)) &= (-1)^m \d \iota_U \Upsilon\\
    &\quad + \sum_{1 \leq i < j \leq m}(-1)^{i+j} \phi \cdot \tau^\ast \iota_{[X_i \wedge X_j] \wedge \cdots\widehat{X}_i \wedge \cdots \wedge \widehat{X}_j \wedge \cdots X_m} \Theta \\
    &= (-1)^m \d \left( \phi  \cdot  \tau^\ast \iota_{X_1 \wedge \cdots \wedge X_m} \Theta \right)\\
    &\quad + \sum_{1 \leq i < j \leq m}(-1)^{i+j} \phi \cdot \tau^\ast \iota_{[X_i \wedge X_j] \wedge \cdots\widehat{X}_i \wedge \cdots \wedge \widehat{X}_j \wedge \cdots X_m} \Theta\\
    &= (-1)^{m+1} \d \left(\phi \cdot \tau^\ast (\alpha_1 \vee \cdots \vee \alpha_m) \right)\\
    &\quad + \sum_{1 \leq i < j \leq m} (-1)^{i+j + 1} \phi \cdot \tau^\ast \left(\{\alpha_i, \alpha_j\} \vee \alpha_1 \cdots \vee \widehat{\alpha}_i \vee \cdots \vee \widehat{\alpha}_j \vee \cdots \vee \alpha_m\right)\,,
\end{align*}
from which the result follows.
\end{proof}

\begin{remark} Notice that the previous result implies that we may find an $L_\infty$-algebra morphism between the Lie algebra $\Omega^{n-1}_H(M, \Theta)$ and the $L_\infty$-algebra $\Omega^{n-1}_H(M, - \d \Upsilon)$:
\[
\widetilde \Psi_m \colon \Omega^{n-1}_H(M, \Theta)^{\otimes m} \longrightarrow \Omega^{n-1}_H(\widetilde M, - \d \Upsilon)\,,
\] in the sense presented in \cite[Proposition 3.8]{CFRZ_16}. Indeed, we may change the sign convention of $\Psi_m$ and $\ell_m$ inductively so that the relation in Theorem \ref{thm:L_infinity} reads as
\begin{align*}
    \widetilde \ell_m (\Psi(\alpha_1), \dotsc, \Psi(\alpha_m)) &= -\d \widetilde \Psi_{m} ( \alpha_1, \dotsc, \alpha_m) \\
    &\quad +\sum_{1 \leq i < j \leq m} (-1)^{i+j +1}\widetilde  \Psi_{m-1}(\{\alpha_i, \alpha_j\}, \alpha_1, \dotsc, \widehat \alpha_i, \dotsc, \widehat \alpha_j, \dotsc, \alpha_m)\,,
\end{align*}
for some $\widetilde \ell_m = \pm \ell_m$ and $\widetilde \Psi_m = \pm \Psi_m$.
\end{remark}

\begin{remark}  In \cite{Vit_15}, Vitagliano introduces higher codimensional versions of contact manifolds that he calls ‘multicontact manifolds’. He introduces an $L_\infty$-algebra associated to a particular distribution of codimension $n$ on a manifold $M$, say $H$. When the line bundle $C \rightarrow M$ defined by
\[
C = \bigg\{\alpha \in \bigwedge^n M\colon H = \{\xi \in \T M \colon \iota_\xi \alpha = 0\}\bigg\} \subseteq \bigwedge^n M
\]
is trivial (to see details we refer to the cited article), the distribution $H$ may be thought as $H = \ker_1 \Theta$, for a nowhere vanishing form $\Theta \in \Omega^{n}(M)$, namely some trivializing section of $C$. Then, Theorem \ref{thm:L_infinity} gives an $L_\infty$-algebra isomorphism between the algebra introduced by Vitagliano and the Lie algebra of conformal Hamiltonian $(n-1)$-forms. Indeed, a symmetry of the distribution $H$ is nothing but a conformal symmetry of $\Theta$.
\end{remark}

\subsection{Multisymplectizations of multicontact manifolds}

Let us first define what a multicontact manifold is. Proposition \ref{prop:multisymplectization} motivates the following:

\begin{definition} 
\label{def:multicontact_manifold}
A \textbf{multicontact manifold} is a pair $(M, \Theta)$ such that 
\[\ker_1 \Theta \cap \ker_1 \d \Theta = \{0\}\qquad\text{and}\qquad \ker_1\d \Theta \neq \{0\}\,.\]
\end{definition}

\begin{remark}

\label{remark:Co-orientable_multicontact}The condition $\ker_1 \d \Theta \neq \{0\}$ ensures that the canonical pre-multisymplectization of a multicontact manifold $(M, \Theta)$ is non-degenerate and, furthermore, it is a property that will be heavily used in the sequel.
\end{remark} 

\begin{remark} Definition \ref{def:multicontact_manifold} recovers the notion of contact form when dealing with $1$-forms. Indeed, let $\eta \in \Omega^1(M)$ such that
\[\ker_1 \eta \cap \ker_1 \d \eta = \{0\} \qquad \text{and} \qquad \ker_1 \d \eta \neq \{0\}\,.\]
First notice that $M$ is necessarily odd dimensional. Indeed, $\dim \ker_1\eta = \dim M - 1$ so that $\dim \ker_1 \d \eta = 1$. Since $\operatorname{codim} \ker_1 \d \eta$ is necessarily even dimensional, $\dim M$ must be odd. Now notice that the map
$$
\begin{array}{rccc}
    \flat\colon & \T M & \longrightarrow & \cT M \\
    & v & \longmapsto & \iota_v{\d \eta} +\eta(v) \cdot \eta
\end{array}
$$
defines a vector bundle isomorphism. Indeed, it is enough to show that it defines a monomorphism of vector bundles. Consider $v \in \T M$ such that
\[\iota_v{\d \eta} +\eta(v) \cdot \eta = 0\,.\]
Contracting again by $v$ yields $\eta(v) ^2 = 0$, which implies $\eta(v)$ = 0. Hence, we have $\iota_v \d \eta = 0$, which gives $v \in \ker_1 \eta \cap \ker_1 \d \eta$ and, hence, $v = 0$. Therefore, $\flat$ defines a vector bundle isomorphism and thus, $\eta$ defines a contact form.
\end{remark}

\begin{remark}
As we mentioned, in \cite{Vit_15}, Vitagliano introduced higher codimensional versions of contact manifolds. This notion is different from ours since, in the cited paper, a multicontact manifold is a manifold equipped with a maximally non-integrable distribution of codimension $n$, which is called a multicontact structure of degree $n$. However, here we are considering differential forms, not distributions. In this sense, our notion generalizes the one introduced in \cite{LGMRR_23, LGMRR_25}. Nevertheless, we may still find a natural connection between the two concepts when the maximally non-integrable distribution arises as the $1$-kernel of a particular $n$-form $\Theta \in \Omega^n (M)$ (see Remark \ref{remark:Co-orientable_multicontact}). Let $\Theta$ be a nowhere vanishing $n$-form on $M$ and define $H := \ker_1 \Theta$. Then if $H$ is maximally non-integrable, we have $\ker_1 \Theta \cap \ker_1 \d \Theta = \{0\}$. The latter condition $\ker_1 \d \Theta \neq \{0\}$ is not necessarily satisfied.
\end{remark}

\begin{remark} A different generalization of contact manifold is that one of $k$-contact manifold (see \cite{ GGMRR_21,  Riv_23}), which (in terms of distributions, see \cite{LRS_24}) is a multicontact manifold of degree $k$ in the sense of Vitagliano $(M, H)$ (see Remark \ref{remark:Co-orientable_multicontact}) that around every point admits $k$ locally defined linearly independent symmetries. When the $k$-contact structure $H$ is both co-orientable and polarizable (see \cite{LRS_24}) we have that $H = \langle \eta_1, \dotsc, \eta_k\rangle^{\circ, 1}$, for certain faimily of $k$ independent $1$-forms $\eta^\mu$. Then, one may find a set of Darboux coordinates around every point on $M$, $(s^\mu, y^j, p^{\mu}_j)$ for $1 \leq \mu \leq n$ and $1 \leq j \leq k$ (notice that this introduces restrictions on the dimension of $M$) such that
\[
\eta^ \mu= \d s^\mu - p^\mu_i \d y^i\,. 
\]
In this scenario, if $k> 1$, we have that the $n$-form $\Theta \in \Omega^n (M \times \mathbb{R}^k)$ defined by
\[
\Theta = \eta^\mu \wedge \d^{n-1} x_\mu = \d s^\mu \wedge \d ^{n-1} x_\mu - p^\mu_i \d y^i \wedge \d^{n-1}x _\mu\,,
\]
where $x^\mu$ denotes the canonical set of coordinates on $\mathbb{R}^k$ and 
\[
\d^{n-1} x_\mu := (-1)^{\mu +1} \d x^1 \wedge \cdots \wedge \widehat{\d x^\mu} \wedge \cdots \wedge \d x^n\,,
\] is a multicontact form, as a quick computation shows.
\end{remark}

The theory of (pre-)multisymplectizations is greatly simplified when dealing with multisymplectizations of multicontact manifolds. First notice that a multicontact $n$-form $\Theta$ is necessarily nowhere vanishing. Indeed, if $\Theta|_p = 0$, then $\ker_1 \Theta |_p = \T _p M$ which would yield $\{0 \}= \ker_1 \Theta |_p \cap \ker_1 \d \Theta |_p = \ker_1 \d \Theta |_p \neq \{0\}$, giving a contradiction.

\begin{proposition}
\label{prop:Liouville_is_vertical} Let $\tau\colon \widetilde M \rightarrow M$ be a fiber bundle with at least $1$-dimensional fibers, where $(\widetilde M, \Omega = -\d (\phi\, \tau^\ast \Theta))$ is a homogeneous multisymplectic manifold for some nowhere vanishing function $\phi \in \Cinfty(\widetilde M)$. Then, its Liouville vector field $\Delta$ is vertical.
\end{proposition}
\begin{proof}
Indeed, denoting $\Upsilon = \phi \cdot \, \tau^\ast \Theta$, the condition $\iota_{\Delta} \Omega = - \Upsilon$ translates into
\begin{equation}
    \d \phi(\Delta) \cdot \tau^\ast \Theta - \d\phi \wedge \iota_{\Delta} \tau^\ast  \Theta  + \phi \cdot \iota_{\Delta} \tau^\ast\d\Theta = \phi \cdot \tau^\ast\Theta\,.
    \label{eq:homoneity_condition}
\end{equation}
Let $p \in \widetilde M$, and let $Y \in \X(\widetilde M)$ be a vertical vector field with $Y|_p \neq 0$. Then, since $\Omega$ is a non-degenerate multisymplectic form, $\iota_{Y} \Omega|_p \neq 0$, that is, $\d \phi(Y)|_p \cdot \tau^\ast \Theta|_p
\neq 0$, and, thus, $\d \phi(Y)|_p \neq 0$. Contract by $Y$ on Eq. \eqref{eq:homoneity_condition} to get $\d \phi(Y) \cdot \iota_{\Delta} \tau^\ast \Theta  = 0$. Since $\d \phi(Y) |_p \neq 0$, we must have $\iota_{\Delta} \tau^\ast \Theta |_p = 0$. Since $p$ is arbitrary, we conclude that $\iota_{\Delta} \tau^\ast \Theta = 0$ everywhere. 

Now, let $Y \in \X(\widetilde M)$ be a projectable vector field such that $\tau_\ast Y \in \ker \d \Theta$ with $\restr{Y}{\tau(p)} \neq 0$. Contracting by it in Eq. \eqref{eq:homoneity_condition} and using that $\iota_\Delta \tau^\ast \Theta = 0$, we get $\d \phi(\Delta) \cdot \tau^\ast  \iota_{\tau_\ast Y} \Theta = \phi \tau^\ast \iota_{\tau_\ast Y} \Theta$ and, since $\tau_\ast Y|_{\tau(p)} \neq 0$ and $\ker_1 \Theta \cap \ker_1 \d \Theta = 0$, we have $\tau_\ast Y|_{\tau(p)} \not \in \ker_1 \Theta$. This implies that $\d \phi(\Delta)|_p = \phi(p)$ and, since $p$ is arbitrary, we conclude that $\phi(\Delta) = \phi$.

Finally, notice that the two previous equalities imply $\iota_\Delta \tau^\ast \d \Theta = 0$ and, together with $\iota_\Delta \tau^\ast \Theta = 0$, we have $\tau_\ast \Delta = 0$, since $\ker_1 \Theta \cap \ker_1 \d \Theta = \{0\}$, concluding the proof.
\end{proof}

\begin{remark} From Remark \ref{remark:Delta_applied_to_phi}, we know that for a multisymplectization $\tau\colon \widetilde M \rightarrow M$ of a multicontact manifold $(M, \Theta)$ with conformal factor $\phi$, we have $\Delta(\phi) = \phi$.
\end{remark}

In fact,
\begin{proposition}
\label{prop:1_dimensional_fibers}
If $\tau\colon \widetilde M \rightarrow M$ is a multisymplectization of the multicontact manifold $(M, \Theta)$, then the fibers are necessarily $1$-dimensional.
\end{proposition}
\begin{proof} Let $Y \in \X(\widetilde M)$ be an arbitrary vertical vector field. We will show that $Y = \frac{Y(\phi)}{\phi} \Delta$. Indeed, notice that 
\[\iota_{Y - \frac{Y(\phi)}{\phi} \Delta} \d (\phi\, \tau^\ast \Theta) = Y(\phi)\tau^\ast\Theta - Y(\phi)\tau^\ast \Theta = 0\,.\]
Now, since the form $- \d (\phi \, \tau^\ast \Theta)$ is non-degenerate, we must have $Y = \frac{Y(\phi)}{\phi} \Delta$, finishing the proof.
\end{proof}

The main ingredient used in the proof of Theorem \ref{thm:multivector_field_lift} was the existence of the connection given in Lemma \ref{lemma:connection}. If $\tau \colon \widetilde M \rightarrow M$ is a multisymplectization of a multicontact manifold $(M, \Theta)$ with conformal factor $\phi \in \Cinfty(\widetilde M)$, then, by Proposition \ref{prop:1_dimensional_fibers}, the connection given by Lemma \ref{lemma:connection} is actually $H = \ker \d \phi$. In particular, we get a well defined horizontal lift:
\[\X^p(M) \longrightarrow \X^p(\widetilde M)\,.\]

Furthermore, we have the following characterization:
\begin{theorem}
\label{thm:Classification_multisymplectization}
Let $(M, \Theta)$ be a connected multicontact manifold and let $\tau\colon \widetilde M \to M$ be a connected multisymplectization with conformal factor $\phi \in \Cinfty(\widetilde M)$. 
Assume that $\ker \d \phi$ defines a complete Ehresmann connection, namely a connection such that the horizontal lift preserves completeness of vector fields. Then, there exist:
\begin{enumerate}[\rm(i)]
    \item a covering space $\pi\colon \Sigma \rightarrow M$;
    \item an open interval $I \subseteq \mathbb{R}$ such that $0 \not \in I$;
    \item a diffeomorphism $\Phi\colon \Sigma \times I \to \widetilde M$ over the identity on $M$ such that $\phi \circ \Phi$ is the projection onto $I$.
\end{enumerate}
\end{theorem}

Before proving the Theorem, let us give a sufficient condition for a local diffeomorphism $\Sigma \rightarrow M$ to define a covering space.

\begin{lemma} 
\label{lemma:covering_space}
Let $\pi\colon \Sigma \to M$ be a surjective local diffeomorphism. Assume that for every path $\gamma\colon [0,1] \rightarrow M$ and for every $x \in \pi^{-1}(\gamma(0))$ there exists a lift of the path $\widetilde \gamma\colon [0,1] \rightarrow \Sigma$ with $\widetilde \gamma(0) = x$. Then, $\pi\colon \Sigma \rightarrow M$ is a covering space.
\end{lemma}
\begin{proof}
Consider $x \in M$ and let $U\ni x$ be an open neighborhood diffeomorphic to $\mathbb{R}^n$. Let $V$ be some connected component of $\pi^{-1}(U)$. Let us prove that 
\[\restr{\pi}{V} \colon V \rightarrow U \]
defines a diffeomorphism. Indeed, it is enough to show that it is bijective.

Let us first prove surjectivity. Take $\widetilde x \in V$ and denote $x := \pi(\widetilde x) \in U$. Let $y \in U$. Since $U$ is path connected, there exists a smooth path $\gamma\colon [0,1] \rightarrow U$ such that $\gamma(0) = x$ and $\gamma(1) = y$. By hypothesis, there exists a lift $\widetilde \gamma \colon [0,1] \rightarrow \Sigma$ such that $\widetilde \gamma(0) = \widetilde x$, and with $\pi \circ \widetilde \gamma = \gamma \circ \pi$. Since $\Ima \widetilde \gamma$ is connected and $\Ima \widetilde \gamma \cap V \neq \emptyset, $ we have $\Ima \widetilde \gamma \subseteq V$, given that $V$ is connected. Now, we have $\pi(\widetilde \gamma (1)) = \gamma( \pi (\widetilde \gamma(1))) = y$, which shows that $\pi|_V\colon V \rightarrow U$ is surjective.

Finally, let us show that it is injective. Suppose there are $\widetilde x_1,\widetilde x_2 \in V$ with $\pi(\widetilde x_1) = \pi(\widetilde x_2)$. Since $V$ is open and connected, it is path connected and there exists a smooth path $\widetilde \gamma\colon [0,1] \rightarrow \Sigma$ with $\widetilde \gamma(0) = \widetilde x_{1}$ and $\widetilde \gamma(1) = \widetilde x_{2}$. This map projects onto a loop $\gamma = \pi \circ \widetilde \gamma$ in $U$. Since, by definition $U$ is diffeomorphic to $\mathbb{R}^n$, then there exists a smooth loop homotopy $\gamma_t$ such that $\gamma_0 = \gamma$ and $\gamma_1$ is a constant path. Lifting these paths to paths $\widetilde \gamma_t$ satisfying $\widetilde \gamma_t(0) = \widetilde x_1$, by continuity we must have $\widetilde \gamma_t(1) = \widetilde x_2$. Taking $t = 1$, since $\gamma_1$ is constant, so is $\widetilde \gamma_1$, which implies $\widetilde x_1 = \widetilde x_2$, finishing the proof. 
\end{proof}

\begin{proof}[Proof of Theorem \ref{thm:Classification_multisymplectization}]
From Proposition \ref{prop:1_dimensional_fibers}, we have that all fibers are necessarily of dimension $1$ and, from Proposition \ref{prop:Liouville_is_vertical}, we conclude that the Liouville vector field $\Delta$ generates $\ker \d \tau$. Define $I:= \Ima \phi$, which is a connected and open subspace of $\mathbb{R}$, namely an open interval. Take $c \in I$ and define $\Sigma:= \phi^{-1}(c)$.

Let us prove that $\pi:= \restr{\tau}{\Sigma} \colon \Sigma \rightarrow M$ is a covering space. Let us first check that it is a surjection. Define
\[U:= \{x \in M \mid c \in \phi(\tau^{-1}(x))\}\qquad\text{and}\qquad V:= \{x \in M \mid c \not \in \phi(\tau^{-1}(x))\}\,.\]
Clearly, $M = U \cup V$, $U \cap V = \emptyset$, $U \neq \emptyset$ and $V$ is open. If we prove that $U$ is open as well, surjectivity will follow from connectedness of $M$. Let $x \in U$ and take a compact neighborhood $K$ of $x$. An arbitrary vector field $X \in \X(M)$ with support contained in $K$ is complete and, using completeness of the Ehresmann connection defined by $\ker \d \phi$, its horizontal lift $X^h$ will be complete as well. Let $\psi_t$, $\psi^h_t$, with $t \in \mathbb{R}$ be the global flows of $X$ and $X^h$, respectively. Since $X^h(\phi) = 0$, the diffeomorphisms $\psi_t^h\colon \widetilde M \rightarrow \widetilde M$ induce diffeomorphisms
\[\restr{\psi_t^h}{\Sigma} \colon \Sigma \rightarrow \Sigma\,.\] 
Noting that $ \pi \circ \psi_t^h = \psi^t \circ \pi$ and that by an adequate choice of a vector field $X$ we can set $\psi^1(x) = y$, for any $y$ in a neighborhood of $x$, we conclude that $U$ is open, and that $\pi$ is surjective. Finally, since the Ehresmann connection defined by $\ker \d \phi$ is complete, $\pi\colon \Sigma \rightarrow M$ satisfies the hypotheses of Lemma \ref{lemma:covering_space} and, therefore, it is a covering space.
Notice that using a similar argument, and the fact that $\widetilde M$ is connected, we can prove that $\phi$ defines a diffeomorphism between any connected component of a fiber $\tau^{-1}(x)$ and $I$.
Finally, define
$$ \begin{array}{rccc}
    \Phi\colon & \Sigma\times I & \longrightarrow & \widetilde M\\
    & (x,t) & \longmapsto & \restr{\phi}{\Sigma_x}^{\!\!\!-1}(t)
\end{array}    
$$
where $\Sigma_x$ is the connected component of $\tau^{-1}(\tau(x))$ that contains $x$. It is clearly bijective. The proof will be finished once we show that it is smooth and that it defines a local diffeomorphism. To show that it is smooth, let us show that its inverse $\Psi := \Phi^{-1}$ is smooth. Indeed, notice that
\[\phi_ \ast \Delta = t \pdv{t},\]
so that the flow of $\Delta$, which we denote by $\psi^{\Delta}_t$, satisfies $\phi(\psi^{\Delta}_t(x)) = e^{t} \phi(x)$. Therefore, the inverse is given by
\[
\begin{array}{rccc}
\Psi\colon & \widetilde M & \longrightarrow & \Sigma \times I\\
& x & \longmapsto & (\psi ^{\Delta}_{\log \abs{c} - \log \abs{\phi(x)}} (x), \phi(x))\,,
\end{array}
\]
which is clearly smooth. To show that it is a local diffeomorphism, notice that for $x \in \widetilde M$ we have
\[(\psi^{\Delta}_t)_\ast \ker \d_x \phi = \ker \d_{\psi^{\Delta}_t(x)}\phi\qquad\text{and}\qquad \phi_\ast \Delta = t \pdv{t}\,.\]
Since $\T M = \ker \d \phi \oplus \langle \Delta \rangle$, we conclude $\Psi_\ast \T M = \T \Sigma \times \T I$, proving that it is a diffeomorphism.
\end{proof} 

More generally, when $\ker \d \phi$ does not define a complete Ehresmann connection, we can prove the following:

\begin{theorem}
\label{thm:covering-multisymplectization}
Let $(M, \Theta)$ be a connected multicontact manifold and $\tau\colon \widetilde M \rightarrow M$ be a connected multisymplectization with conformal factor $\phi \in \Cinfty(\widetilde M)$. Assume that for every smooth curve $\gamma\colon[0,1] \rightarrow M$ there exists a smooth lift $\widetilde \gamma \colon [0,1] \rightarrow \widetilde M$ passing through any point in the fiber $\tau^{-1}(\gamma(0))$. Then, there exist
\begin{enumerate}[\rm(i)]
    \item a covering space $\pi\colon \Sigma \rightarrow M$;
    \item an open interval $I \subseteq \mathbb{R}$ such that $0 \not \in I$;
    \item an open embedding $\Psi\colon \widetilde M \to \Sigma \times I$ over the identity on $M$ such that the map $\restr{\phi \circ \Psi^{-1}}{\Ima \Psi}$ is the projection onto $I$.
\end{enumerate}
\end{theorem}

Before proving this result, we need the following lemma.

\begin{lemma} Define $\mathcal{F}$ to be the foliation defined by  $\langle \Delta\rangle$, the distribution generated by the Liouville vector field. Then, there exists a smooth manifold structure on $\Sigma = \widetilde M / \mathcal{F}$ so that the canonical projection $p\colon\widetilde M \rightarrow \Sigma$ defines a submersion.
\end{lemma}
\begin{proof}
First notice that, since $\phi$ is nowhere vanishing and $\Delta(\phi) = \phi$ (see Remark \ref{remark:Delta_applied_to_phi}), we have that $\phi$ defines a submersion onto its image. Let $x \in \widetilde M$ and $\Sigma_x := \phi^{-1}(\phi(x))$, which is a submanifold since $\phi(x)$ is a regular value of $\phi$. Notice that the restriction $\restr{p}{\Sigma_x} \colon \Sigma_x \rightarrow \Sigma$ defines an homeomorphism between an open neighborhood of $x$ in $\Sigma_x$, $U_x$, and an open neighborhood of $p(x)$ in $\Sigma$, $V_x$. Let $\varphi_x \colon \mathbb{R}^n \rightarrow U_x$ be a parametrization of $U_x$ and define the induced parametrization of $V_x$ as $\widetilde \varphi_x := p \circ \varphi_x$.
We will show that the collection $\mathcal{Z} := \{(V_x, \widetilde \varphi_x), x \in \widetilde M\}$ defines an atlas on $\Sigma$. Indeed, let us show that coordinate changes are smooth. Consider $y \in \widetilde M$ such that $V_y \cap U_x \neq \emptyset$. Now, since $\varphi_\ast \Delta = t \pdv{t}$, denoting by $\psi^{\Delta}_t$ the flow of $\Delta$, we have $\phi(\psi^\Delta_t(x)) = e^t \phi(x)$. If $\varphi_x$ is a parametrization of $\Sigma_x$ on $(p|_{\Sigma_x})^{-1}(V_x\cap V_y)$, then $\psi^{\Delta}_{\log \abs{\phi(y)} - \log\abs{\phi(x)}} \circ \varphi_x$ is a parametrization of $\Sigma_y$ on $(p|_{\Sigma_y})^{-1}(V_c \cap V_y)$. Therefore, if $\varphi_y$ is a different parametrization, the coordinate change $\widetilde \varphi_y^{-1} \circ \widetilde \varphi_x = \varphi^{-1}_y \circ \psi^{\Delta}_{\log \abs{\phi(y)} - \log\abs{\phi(x)}} \circ \varphi_x $ is smooth, proving the claim. Furthermore, from the construction, it is easy to see that $p$ defines a submersion.
\end{proof}

\begin{proof}[Proof of Theorem \ref{thm:covering-multisymplectization}]
Since the distribution generated by $\Delta$ is $1$-dimensional, it is involutive. Let $\mathcal{F}$ denote the associated foliation and define \[
\Sigma := \widetilde M / \mathcal{F}\,.
\] 
Notice that $\tau_\ast \langle\Delta \rangle = \{0\}$ (since $\Delta$ is vertical) so that we have an induced map $\pi\colon \Sigma \rightarrow M$, which is a local diffeomorphism. Furthermore, since $\widetilde M$ has the path lift property, so does $\Sigma$ (by projection) and, therefore, using Lemma \ref{lemma:covering_space}, $\pi\colon \Sigma \rightarrow M$ defines a covering space.
Finally, it is enough to define
\[
\begin{array}{rccc}
\Phi \colon & \widetilde M & \longrightarrow & \Sigma \times I\\ 
& x & \longmapsto & (p(x), \phi(x))\,,
\end{array}
\]
which is clearly an embedding.
\end{proof}


\section{The \texorpdfstring{$\sharp$}{} mapping}

\label{section:sharp_mapping}

Every contact manifold $(M, \eta)$ has an underlying Jacobi structure, namely a pair $(\Lambda, E)$ consisting of a bivector and vector field, respectively, such that
\[[\Lambda, \Lambda] = 2 E \wedge \Lambda\qquad\text{and}\qquad [\Lambda, E] = 0\,.\]
Then, there is an induced map $\sharp_\Lambda \colon\cT M \rightarrow \T M$ given by $\sharp_\Lambda(\alpha) = \iota_{\alpha} \Lambda$
and the Jacobi bracket of two functions $f, g \in \Cinfty(M)$ can be written as 
\begin{equation}
\label{eq:Jacobi_bracket_contact}
    \{f, g\} = \iota_{\sharp_{\Lambda}(\d f)} \d g + E(g) f - E(f) g\,.
\end{equation}
The Jacobi structure may be interpreted as a linearized version of the Jacobi brackets on a contact manifold, and studying one or the other are complete equivalent. Consequently, recent research has been devoted to study the equivalent analogue of Poisson and Dirac structures from classical mechanics to classical field theories (see \cite{BMR_17, CFRZ_16, LI_25, Z_12}), thus looking for a linearized version of the (now graded) Poisson brackets. The vector bundle map $\sharp_\Lambda$ is generalized to a vector bundle map
\[\sharp \colon \mathcal{Z}^n \longrightarrow \T M,\]
where $\mathcal{Z}^n$ is a vector subbundle of $\bigwedge^n M$ and $\sharp$ is skew-symmetric, namely 
\[\iota_{\sharp(\alpha)} \beta = -\iota_{\sharp(\beta)} \alpha\,,\]
for every $\alpha, \beta \in \mathcal{Z}^n$. In this section we define the analogue of the previous $\sharp$ mapping when dealing with a multicontact manifold $(M, \Theta)$, where $\Theta$ has now arbitrary order $n$. We draw inspiration from the theory of Dirac--Jacobi bundles, a common generalization of Jacobi and pre-contact manifolds introduced by Vitagliano in \cite{Vit_2018}. Most notably, instead of dealing with subbundles $\mathcal{Z}^n \subseteq \bigwedge^n M$, we need to work with subbundles 
\[
\mathcal{Z}^n \subseteq \bigwedge^{n-1}M \oplus \bigwedge^{n} M
\]
and maps 
\[
\sharp \colon \mathcal{Z}^n \longrightarrow \T M \oplus \mathbb{R}\,.
\]

 Let us first introduce some elementary operations on the bundles of pairs of forms and multivector fields
\[
\bigwedge^aM \oplus \bigwedge^{a+1} M\,, \qquad \bigvee_p M \oplus \bigvee_{p-1} M\,.
\]
For the particular cases $p = 0$ and $p = 1$ we understand $\bigvee_{0} M = M \times \mathbb{R}$ and $\bigvee_{-1} M = M \times \{0\}$. In these bundles there is a well defined contraction, namely for $(\alpha, \beta) \in \bigwedge^{a}M \oplus \bigwedge^{a+1} M$ and $(U, V) \in \bigvee_p M \oplus \bigvee_{p-1} M$, we set
\[
\iota_{(U, V)}(\alpha, \beta) := (\iota_U \alpha, \iota_U \beta + (-1)^{a}\iota_V \alpha) \in \bigwedge^{a-p} M \oplus \bigwedge^{a+1-p} M\,.
\]
Analogously, we define the wedge product, for $(U_1, V_1) \in \bigvee_p M \oplus \bigvee_{p-1} M$, and $(U_2, V_2) \in \bigvee_q M\oplus \bigvee_{q-1} M$
\[
(U_1, V_1)\wedge (U_2, V_2) := (U_1 \wedge U_2, V_1\wedge U_2 + (-1)^{pq} V_2 \wedge U_1) \in \bigvee_{p+q} M \oplus \bigvee_{p+q-1} M \,.
\]

\begin{proposition}
\label{prop:properties_of_pairwise_product}
The wedge product and contraction defined above satisfy the following equalities:
\begin{enumerate}[\rm (i)]
    \item For $(U_1, V_1) \in \bigvee_p M \oplus \bigvee_{p-1} M$ and $(U_2, V_2) \in \bigvee_q M \oplus \bigvee_{q-1} M$ we have 
    \[
    (U_1, V_1) \wedge (U_2, V_2)= (-1)^{pq} (U_2, V_2) \wedge (U_1, V_1)\,.
    \]
    \item For $(U_1, V_1) \in \bigvee_p M \oplus \bigvee_{p-1} M$, $(U_2, V_2) \in \bigvee_q M \oplus \bigvee_{q-1} M$ and $(\alpha, \beta) \in \bigwedge^a M \oplus \bigwedge^{a+1} M$, we have 
    \[
    \iota_{(U_2, V_2)}\iota_{(U_1, V_i)} (\alpha, \beta) = \iota_{(U_1, V_1) \wedge (U_2, V_2)}(\alpha, \beta)\,.
    \]
\end{enumerate}
\end{proposition}
\begin{proof} Both statements follow from a straightforward computation.
\end{proof}


Let $\mathcal{Z}^{n+1} := \langle \Theta \rangle \oplus \langle \d \Theta \rangle \subseteq \bigwedge^{n} M \oplus \bigwedge^{n+1} M$ and define for $a = 1 ,\dotsc, n$, 
\begin{align*}
    \mathcal{Z}^a &:= \iota_{\bigvee_{n+1-a} M \oplus \bigvee_{n-a} M}\left( \mathcal{Z}^{n+1}\right)\\
    &:= \left\langle \iota_{(U, V)}(\Theta, \d \Theta) : (U, V) \in \bigvee_{n+1-a}M \oplus \bigvee_{n-a} M \right\rangle
\end{align*}
We have the following description of the previous subspaces
\begin{lemma}
\label{lemma:description_of_Za}
For $a = 1, \dotsc, n$, we have
\[
\mathcal{Z}^a = \left\{ \left(\iota_U \Theta, \iota_U \d \Theta + \iota_V \Theta\right) \colon\, U \in \bigvee_{n+1-a} M,\, V \in \bigvee_{n-a} M\right\}\,,\]
and, for $a  = 1, \dotsc, n-1$:
\[
\mathcal{Z}^a = \iota_{\T M \oplus \mathbb{R}} \mathcal{Z}^{a+1} = \left\langle\iota_{(X, r)} (\alpha, \beta) \colon (\alpha, \beta) \in \mathcal{Z}^{a+1}\,, X \in \T M\,, r \in \mathbb{R}  \right\rangle\,.
\]
\end{lemma}
\begin{proof} The first equality follows easily from the definition:
\begin{align*}
    \mathcal{Z}^a &=  \langle \iota_{(U, V)}(\Theta, \d \Theta) : (U, V) \in \bigvee_{n+1-a}M \oplus \bigvee_{n-a} M \rangle\\
    &=  \langle (\iota_U \Theta, \iota_U\d \Theta + (-1)^{a} \iota_V \Theta : (U, V) \in \bigvee_{n+1-a}M \oplus \bigvee_{n-a} M \rangle\,.
\end{align*}
The second equality follows similarly, particularizing in the case where both $U$ and $V$ are decomposable.
\end{proof}

We will assume throughout the rest of the section that the subspaces $\mathcal{Z}^a$ define subbundles, for every $a = 1, \dotsc, n$.

\begin{remark} One immediate application of the introduction of the previous subbundles is the following: an $a$-form $\alpha \in \Omega^{a}(M)$ is conformal Hamiltonian if and only if $(- \alpha, (-1)^{n +1-a}\d \alpha) \in \mathcal{Z}^{a+1}$. Indeed, in light of Eq. \eqref{eq:Hamiltonian_form}, we have 
\begin{align*}
    - \alpha &= \iota_{X_\alpha} \Theta\,,\\
    (-1)^{n+1-a}\d \alpha &= \iota_{X_\alpha } \d \Theta + (-1)^{n-a} \iota_{V_\alpha } \Theta \,.
\end{align*}
\end{remark}

\begin{remark} One may remove the signs in the condition $(- \alpha, (-1)^{n+1-a} \d \alpha) \in \mathcal{Z}^a$ by introducing a different sign convention in the definition of the subbundles. We prefer to work with the presented convention, as it simplifies the computations to come. 
\end{remark}

We are now ready to introduce the definition of the $\sharp$-mapping:

\begin{definition} The {\bf $\widehat \sharp$-mapping} is the unique vector bundle map 
\[
\widehat \sharp \colon \mathcal{Z}^n \longrightarrow \T M \oplus \mathbb{R}
\]
satisfying $\widehat\alpha = \iota_{\widehat{\sharp}(\alpha)} (\Theta, \d \Theta)$, for every $\widehat\alpha \in \mathcal{Z}^{n}$.
\end{definition}

\begin{proposition} The $\widehat \sharp$-mapping is well defined.
\end{proposition}
\begin{proof} The proof follows if we show that if for certain $X \in \T M$ and $r \in \mathbb{R}$ we have 
\[
\iota_X \Theta =0 \,,\qquad \iota_X \d \Theta + r  \Theta = 0\,,
\]
then $X = 0$ and $r = 0$. Indeed, recall that $\ker_1 \d \Theta \neq 0$ so, by contracting by an element $ 0 \neq Y \in \ker_1 \d \Theta$ in the second equation we have $r\,\iota_Y\Theta = 0$. Since $\ker_1 \Theta \cap \ker_1 \d \Theta = \{0\}$, we necessarily have $Y \not \in \ker_1 \Theta$, so that $r = 0$. Therefore, $\iota_X \Theta = 0$ and $\iota_X \d \Theta = 0$. Using again $\ker_1 \Theta \cap \ker_1 \d \Theta = \{0\}$, we conclude $X = 0$.
\end{proof}

\begin{remark} Notice that for a conformal Hamiltonian $(n-1)$-form $\alpha \in \Omega^{n-1}_H(M)$, we have 
\[
\widehat{\sharp} (- \alpha, \d  \alpha) = (X_\alpha, (-1)^{n-1} V_\alpha)\,,
\]
where $V_\alpha \in \Cinfty(M)$ is the conformal factor. 
\end{remark}

Hence, we may recover both the conformal Hamiltonian vector field and the conformal factor using the $\widehat \sharp $-mapping. Therefore, it is convenient to find a generalization to degrees $a < n$. First notice that $\widehat \sharp$ is {\it skew-symmetric}, in the following sense:
\begin{proposition} 
\label{prop:skew_symmetry}
For every $\alpha \in \mathcal{Z}^n$ and $ \beta \in \mathcal{Z}^n$ we have 
\[
\iota_{\widehat \sharp (\alpha)} \beta = - \iota_{\widehat\sharp (\beta)} \alpha\,.
\]
\end{proposition}
\begin{proof} Indeed, it is enough to show $\iota_{\widehat{\sharp}(\alpha)} \alpha = 0$, which follows from 
\[
\iota_{\widehat{\sharp}(\alpha)} \alpha = \iota_{\widehat{\sharp}(\alpha)} \iota_{\widehat{\sharp}(\alpha)} (\Theta, \d \Theta) = \iota_{\widehat{\sharp}(\alpha) \wedge \widehat{\sharp}(\alpha)} (\Theta, \d \Theta) = 0\,,
\]
where in the last two equalities we have used Proposition \ref{prop:properties_of_pairwise_product}.
\end{proof}

Now, to obtain the graded version of $\widehat{\sharp}$, we need to introduce the following subbundles. For $p = 1, \dotsc, n$, define
\[
K_p := (\mathcal{Z}^p)^{\circ, p} =\left\{ U \in \bigvee_{p}M \oplus \bigvee_{p-1} M \colon\,\, \iota_U \alpha = 0\,, \forall \alpha \in \mathcal{Z}^p \right\}\,.\]
Now we may obtain the $\widehat \sharp$ maps as follows:

\begin{definition}
\label{def:graded_sharp_maps}
By Lemma \ref{lemma:description_of_Za}, we know that $\mathcal{Z}^a$ is generated by elements of the type $\iota_{\widehat U} \widehat \alpha$, with $\widehat U \in \bigvee_{n-a} M \oplus \bigvee_{n-1-a} M$  and $\widehat \alpha \in \mathcal{Z}^a$. We define the extensions \[
\widehat{\sharp}_a \colon \mathcal{Z}^a \longrightarrow \left(\bigvee_{n+1-a} M \oplus \bigvee_{n-a} M \right)  \bigg / K_{n+1-a}
\] on such elements as
$
\widehat{\sharp}_a(\iota_U \alpha) = \widehat{\sharp}(\alpha) \wedge U + K_{n+1-a}\,,
$
and extend by linearity to $\mathcal{Z}^a$.
\end{definition}

We only need to check well-definedness:

\begin{proposition}
\label{prop:well_definedness_of_sharp}
The $\widehat \sharp$-maps of Definition \ref{def:graded_sharp_maps} are well defined.
\end{proposition}

\begin{proof} Indeed, suppose that there was a sequence of pairs $U_i \in \bigvee_{n-a}M \oplus_M \bigvee_{n-1-a} M$ and $\alpha^i \in \mathcal{Z}^n$ such that $\iota_{U_i} \alpha^i = 0$ (sum intended), then we will show that 
\[
\widehat{\sharp}(\alpha^i) \wedge U_i \in K_{n+1-a}\,, 
\]
which will end the proof. Indeed, let $\iota_V\beta \in \mathcal{Z}^{n+1-a}$, for certain $V \in \bigvee_{a-1} M \oplus \bigvee_{a-2} M$ and $\beta \in \mathcal{Z}^a$. Using the properties of Proposition \ref{prop:properties_of_pairwise_product} and the skew-symmetry of Proposition \ref{prop:skew_symmetry} we have
\begin{align*}
    \iota_{\widehat{\sharp} (\alpha^i) \wedge U_i} (\iota_V\beta) &= \iota_{U_i} \iota_{\widehat{\sharp}(\alpha^i)} (\iota_V\beta) = (-1)^{a-1} \iota_{U_i} \iota_V \iota_{\widehat{\sharp}(\alpha^i)} \beta \\
    &= (-1)^{a} \iota_{U_i} \iota_V \iota_{\widehat{\sharp}(\beta)} \alpha^i = (-1)^{(n+1-a)\cdot a } \iota_{V} \iota_{\widehat{\sharp}(\beta)}\left( \iota_{U_i} \alpha^i\right)\\
\end{align*}
which vanishes by hypotheses and shows that $\widehat{\sharp}(\alpha^i) \wedge U_i \in K_{n+1-a}$, as $\beta \in \mathcal{Z}^{n}$, $V \in \bigvee_{a-1} M \oplus \bigvee_{a-2}$ is arbitrary, and $\iota_V \beta$ generates $\mathcal{Z}^{n+1-a}$, finishing the proof.
\end{proof}

\begin{remark} Notice that 
\begin{align*}
    K_p &= \bigg\{(U, V) \in \bigvee_p M \oplus \bigvee_{p-1} M \colon \iota_{(U, V)}(\Theta, \d \Theta) = 0\bigg\}\\
    &= \bigg\{(U, V) \in \bigvee_p M \oplus \bigvee_{p-1} M \colon U \in \ker_p \Theta\,, \iota_U \d \Theta = (-1)^n \iota_V \Theta\bigg\},
\end{align*} and then, given the description of the $\widehat{\sharp}$-mappings above, for every $\widehat \alpha \in \mathcal{Z}^a$, we have $\widehat \alpha = \iota_{\widehat{\sharp}_a(\alpha)} (\Theta, \d \Theta)$
\end{remark}

\begin{remark} There is a well defined contraction $\iota_{\left[(U, V)\right]} \widehat \alpha$, for \[\left[(U, V)\right] \in \left(\bigvee_{n+1-p} M \oplus \bigvee_{n-p} M \right)\bigg / K_p\] and $\widehat \alpha \in \mathcal{Z}^a$, with $p \leq a$. Indeed, for pairs $(U, V) \in K_p$ we have $\iota_{(U, V)} \widehat \alpha = 0$, as a quick comprobation shows (using Lemma \ref{lemma:description_of_Za} {\rm (ii)}).
\end{remark}

This last remark implies that we also get a graded-skew-symmetry generalizing Proposition \ref{prop:skew_symmetry}. Indeed, the proof is a similar computation to that performed in Proposition \ref{prop:well_definedness_of_sharp}:

\begin{proposition}
\label{prop:graded_skew_symmetry}
For $\alpha \in \mathcal{Z}^a$ and $\beta \in \mathcal{Z}^b$ we have 
\[
\iota_{\widehat{\sharp}_a(\alpha)} \beta = (-1)^{(n+1-a)(n+1-b)} \iota_{\widehat{\sharp}_b(\beta)} \alpha \,.
\]
\end{proposition}

\begin{proof} By using the description given in Definition \ref{def:graded_sharp_maps}, it is suffices to show it for decomposable elements $\iota_U \alpha$, and $\iota_V \beta$ for certain $U \in \bigvee_{n-a} M \oplus \bigvee_{n-1-a} M$, $V \in \bigvee_{n-b} M \oplus \bigvee_{n-1-b} M$ and $\alpha, \beta \in \mathcal{Z}^n$. Indeed, again using Proposition \ref{prop:properties_of_pairwise_product} and Proposition \ref{prop:skew_symmetry},
\begin{align*}
    \iota_{\widehat{\sharp}_a(\iota_{U} \alpha)} (\iota_V\beta) &= \iota_{\widehat{\sharp}(\alpha) \wedge U} (\iota_V\beta) = \iota_U \left( \iota_{\widehat{\sharp}(\alpha)}(\iota_V \beta)\right)\\
    & = (-1)^{n-b} \iota_U \iota_V \iota_{\widehat{\sharp}(\alpha)} \beta= (-1)^{n-1-b} \iota_U \iota_V \iota_{\widehat{\sharp}(\beta)} \alpha \\
    & = (-1)^{n-1-b}  \iota_U \iota_{\widehat{\sharp}(\beta) \wedge V} \alpha\\
    &= (-1)^{n+1-b + (n-a)\cdot(n+1-b)} \iota_{\widehat{\sharp}(\beta) \wedge V} \left(  \widehat{\sharp}(\alpha) \wedge U\right)\\
    &= (-1)^{(n+1-a)\cdot (n+1-b)} \iota_{\widehat{\sharp}_{b}(\iota_V \beta)} \left( \iota_U \alpha\right)\,,
\end{align*}
which proves the result.
\end{proof}

\begin{remark} Actually, using some straight-forward computations, we may show that the formula given in Definition \ref{def:graded_sharp_maps} is actually the unique possible family of maps 
\[
\widehat{\sharp}_a \colon \mathcal{Z}^a \longrightarrow \left(\bigvee_{n+1-a} M \oplus \bigvee_{n-a} M \right)  \bigg / K_{n+1-a}
\]
that satisfy the graded-skew-symmetry of Proposition \ref{prop:graded_skew_symmetry}. Indeed, if it satisfies graded skew symmetry, for $\alpha \in \mathcal{Z}^a$, $U \in \bigvee_{n-a} M \oplus \bigvee_{n-1-a} M$ and $\beta \in \mathcal{Z}^{n+1-a}$, we should have 
\begin{align*}
    \iota_{\widehat{\sharp}_a(\iota_{U} \alpha)} \beta &= (-1)^{(n-1-a)\cdot a} \iota_{\widehat{\sharp}_{n+1-a}(\beta)} \iota_U \alpha = (-1)^{(n+1-a)\cdot a + (n-a)\cdot a} \iota_U \iota_{\widehat{\sharp}_{n+1-a}(\beta)}\alpha \\
    &= (-1)^{a}\iota_U \iota_{\widehat{\sharp}_{n+1-a}(\beta)}\alpha =  \iota_U  \iota_{\widehat{\sharp}(\alpha)} \beta = \iota_{\widehat{\sharp}(\alpha) \wedge U} \beta\,,
\end{align*}
which implies 
\[
\widehat{\sharp}_a(\iota_{U} \alpha) = \widehat{\sharp}(\alpha) \wedge U + K_{n+1-a}\,,
\]
as the equality holds for arbitrary $\beta \in \mathcal{Z}^{n+1-a}$. In fact, this is a generalization of the study some of the authors performed in \cite{LI_25}.
\end{remark}

Now that we have the graded version of the $\widehat{\sharp}$-mapping, we can obtain the Hamiltonian multivector field and conformal factor of an arbitrary conformal Hamiltonian form $\alpha \in \Omega^{a}_H(M)$ as 
\[
\widehat{\sharp}_{a+1}(- \alpha, (-1)^{n+1-a} \d \alpha) = (X_\alpha, (-1)^a V_\alpha) + K_{n-a}\,.
\]
Also, by introducing the following generalization of the Schouten--Nijenhuis bracket for pairs $(U_1, V_1) \in\X^p(M) \oplus \X^{p-1}M$ and $(U_2, V_2) \in \X^q (M) \oplus \X^{q-1}(M)$:
\[
[(U_1, V_1), (U_2, V_2)] := ([U_1, U_2], (-1)^{p-1} [U_1, V_2] - (-1)^{p (q-1)} [U_2, V_1]) \in \X^{p+q-1}(M) \oplus \X^{p+q-2}(M)\,,
\]
we get have that and we may write the bracket in terms of these $\widehat{\sharp}$-mappings, and that $\widehat \sharp$ maps the previous Schouten--Nijenhuis bracket into the graded Jacobi bracket

\begin{theorem} 
\label{thm:bracket_formula_sharp}
Let $\alpha \in \Omega^a_H(M)$ and $\beta \in \Omega^b_H(M)$ be conformal Hamiltonian forms. Then the following holds:
\begin{enumerate}[\rm (i)]
    \item We may express the bracket with the $\widehat \sharp$-maps as:
    \[
(\alpha \vee \beta, (-1)^{a} \{\alpha, \beta\} ) = \iota_{\widehat{\sharp}_{b+1}(- \beta, (-1)^{n+1-b} \d \beta)} (- \alpha, (-1)^{n+1-a} \d \alpha) + (0, (-1)^{a+b}\d \left(\alpha \vee \beta\right))\,.
\]
    \item The $\widehat \sharp$-map satisfies the following bracket preserving formula:
    \[
\widehat{\sharp}_{a+b - (n-1)}(- \{\alpha, \beta\}, (-1)^{a+b}\d \{\alpha, \beta\}) = \left[\widehat{\sharp}_a(- \alpha, (-1)^{n+1-a} \d \alpha), \widehat{\sharp}_b(- \beta, (-1)^{n+1-b} \d \beta) \right]\,.
\]
\end{enumerate}

\end{theorem}
\begin{proof} Using Eq. \eqref{eq:Hamiltonian_form}, both follow from a straight-forward computation.
\end{proof}

\section{Multicontact field equations}
\label{section:Field_equations}

In this section we study a possible definition of multicontact field equations in a particular, but still quite general, subclass of multicontact manifolds. We look for equations that not only are dependent on the geometry, but also depend on a choice of Hamiltonian, as is usually done in symplectic and contact mechanics. In contact mechanics, the evolution equation of the motion determined by a Hamiltonian $H \in \Cinfty(M)$ on a contact manifold $(M, \eta)$ is given by 
\begin{equation}
\label{eq:Contact_evolution}
    \dot{g} = X_H(g) = \{H, g\} - {R}(H) g\,,
\end{equation}
where $R$ is the Reeb vector field defined by $\eta$ and $\{H, g\}$ denotes the usual Jacobi bracket. 

In our case, a Hamiltonian will be a suitable $n$-form $h \in \Omega^{n}(M)$, and the evolution of conformal Hamiltonian $(n-1)$-forms will be given by 
\begin{align*}
       \psi^\ast({\d \alpha}) &= -\mathcal{R}(\d \alpha) \cdot h - \iota_{X_\alpha} \d h - \left(\iota_{\mathcal{R}} \d h\right) \wedge \alpha\\
       &= - \left(\Lie_{X_\alpha} + \mathcal{R}(\alpha) \right) h - (\iota_{\mathcal{R}} \d h) \wedge \alpha +\d \left(\iota_{X_\alpha} h \right),
\end{align*}
where $\mathcal{R}$ is the {\it Reeb multivector field}, a generalization of the Reeb vector field, and $\psi: X \rightarrow M$ is a smooth map from an $n$-dimensional manifold to the multicontact manifold $(M, \Theta)$. Using Remark \ref{remark:Expression_for_Jacobi_bracket}, extending the domain of definition of brackets to $n$-forms, the previous equation may be interpreted as 
\[\psi^{\ast}(\d \alpha) = \{h, \alpha\}- (\iota_{\mathcal{R}} \d h) \wedge \alpha +\d \left(\iota_{X_\alpha} h \right)\,,\]
which clearly generalizes Eq. \eqref{eq:Contact_evolution}.

All over this section we will assume that $\ker_1 \d \Theta$ defines a vector subbundle of constant rank.

\subsection{General multicontact Hamilton--de Donder--Weyl equations and variational multicontact manifolds}

{Throughout this subsection we assume $n > 1$.} {Let us first define the Reeb multivector field. One first attempt would be to set $\mathcal{R}$ as the unique $n$-multivector field up to $\ker_n \Theta \cap \ker_n \d \Theta$ such that
$\iota_{\mathcal{R}} \Theta = 1$ and $\iota_{\mathcal{R}} \d \Theta = 0$.
However, a finer definition can be made. Let $\flat_\Theta$ denote contraction by $\Theta$, namely
\[
\begin{array}{rccc}
    \flat_\Theta\colon & \ker_1 \d \Theta & \longrightarrow & \bigwedge^{n-1} M\\
    & v & \longmapsto & \iota_v \Theta\,.
\end{array}
\]
Notice that it defines a monomorphism. Then, we may interpret its inverse $\flat_\Theta^{-1} \colon \Ima \flat_\Theta\rightarrow \ker_1 \dd \Theta$ as a tensor
\[
\flat_\Theta^{-1} \in \ker_1 \d \Theta \otimes (\Ima \flat_\Theta)^\ast \cong \ker_1 \d \Theta \otimes  \left(\bigvee_{n-1}M \Big / (\Ima \flat_\Theta)^{\circ, n - 1}  \right)\,.
\]

 \begin{definition} Define the {\bf Reeb multivector field} as $ \widetilde{\mathcal{R}} := \frac{1}{n} \flat_\Theta ^{-1}$, interpreted as a tensor taking values in the vector bundle above.
 \end{definition}
 
Proposition \ref{prop:relation_of_multivectors} will explain why $\widetilde{\mathcal{R}}$ is a finer choice than defining $\mathcal{R}$ modulo $\ker_n \Theta \cap \ker_n \d \Theta$ by the equations $\iota_{\mathcal{R}} \Theta = 1$ and $\iota_{\mathcal{R}} \d \Theta$, which will also justify the choice for the name {\it multivector}. First, notice that we have a well defined mapping 
\[
\begin{array}{rccc}
\phi\colon & \ker_1 \d \Theta \otimes\left( \bigvee_{n-1}M \Big / (\Ima \flat_\Theta)^{\circ, n - 1} \right) & \longrightarrow & \bigvee_{n}M/ K_n \\
& R \otimes (U + (\Ima \flat_\Theta)^{\circ, n - 1}) & \longmapsto & R \wedge U + K_n\,,
\end{array}
\]
where $K_n = \ker_n \Theta \cap \ker_n \d \Theta$, since $\ker_1 \d \Theta \wedge (\Ima \flat_{\Theta})^{\circ, n - 1} \subseteq K_n$.
\begin{proposition}
\label{prop:relation_of_multivectors}
We have $\phi(\widetilde {\mathcal{R}}) = \mathcal{R}$.
\end{proposition}
\begin{proof} Let $R_i$ denote a basis for $\ker_1 \d \Theta$ and define $\alpha_i := \iota_{R_i} \Theta$. Then, $\widetilde{\mathcal{R}} = \frac{1}{n} R_i \otimes u^i$ and we have $\phi(\widetilde{\mathcal{R}}) = \frac{1}{n}R_i \wedge u^i$ which clearly satisfies
$$\iota_{\phi(\widetilde{\mathcal{R}})} \d \Theta = 0\qquad\text{and}\qquad \iota_{\phi(\widetilde{\mathcal{R}})} \Theta = 1\,. $$
\end{proof}
}

We now aim to define contraction by $\widetilde{\mathcal{R}}$ on a suitable family of $n$-forms, which will be later identified as the Hamiltonians, objects that will define the dynamics.

\begin{remark}
\label{remark:well_definedness}
Let \[\widetilde U = U \otimes [V + (\Ima \flat_{\Theta})^{\circ, n+1}] \in  \ker_1 \d \Theta \otimes\left( \bigvee_{n-1}M \Big / (\Ima \flat_\Theta)^{\circ, n - 1} \right)\,.\] For a suitable $a$-form $\alpha$, we want to define the contraction $\iota_{\widetilde U} \alpha$ as $\iota_V\iota_{U} \alpha$. We can guarantee well-definedness when $\alpha$ is such that $\iota_{K} \iota_u \alpha = 0$, for every $K \in (\Ima \flat_\Theta)^{\circ, n-1}$ and every $u \in \ker_1 \d \Theta$. This last condition is equivalent to $\alpha$ satisfying 
$\iota_V \iota_u \alpha \in \Ima \flat_\Theta$, for every $u \in \ker_1 \d \Theta$ and $V \in \bigvee_{a-n}M$, as a simple exercise in multilinear algebra shows.
\end{remark}

\begin{definition}
\label{def:Hamiltonian}
Define the \textbf{Hamiltonian subbundle} $\mathcal{H} \subseteq \bigwedge^n M$ as 
\[\mathcal{H} = \bigg\{h \in \bigwedge^{n}M \ \Big\vert\ \iota_{\T M} h \subseteq \Ima \flat_{\Theta}\bigg\}\,.\]
A \textbf{Hamiltonian} is a section $h \in \Gamma(\mathcal{H})$, that is, an $n$-form $h$ such that $\iota_{\T M} h \subseteq \Ima \flat_\Theta$.
\end{definition}

Now, let $h \in \Gamma(\mathcal{H})$ be a Hamiltonian. In order to define the dynamics, we will need to evaluate $\iota_{\widetilde{\mathcal{R}}} \d h$, where the contraction is obtained with any representative of ${\widetilde{\mathcal{R}}}$ in $\ker_1 \d \Theta \otimes \bigwedge_{n-1} M$. Remark \ref{remark:well_definedness} motivates the following definition:

\begin{definition} A Hamiltonian $h \in \Gamma(\mathcal{H})$ is called a {\bf good Hamiltonian} if $\iota_{\ker_1 \d \Theta} \d h \subseteq \mathcal{H}$.
\end{definition}

For a good Hamiltonian, a fundamental object is the \textbf{dissipation $1$-form} which is defined as \[\sigma_h := {(-1)^{n-1}} n \cdot \iota_{\widetilde{\mathcal{R}}} \d h\,.\]

\begin{remark} Both the sign and the factor of $n$ are introduced to recover the original definition found in \cite{LGMRR_23} (see Lemma \ref{lemma:formula_for_dissipation_form}).
\end{remark}

Then, we can define the \dfn{Hamilton--de Donder--Weyl} equations for a map $\psi\colon X \rightarrow M$, where $X$ is and $n$-dimensional manifold, as
\begin{equation*}
    \psi^{\ast} (\Theta + h) = 0\qquad\text{and}\qquad 
    \psi^{\ast} \iota_{\xi} (\d + \sigma_h \wedge)(\Theta + h) =  0\,,\quad \forall \xi \in \X(M)\,.
\end{equation*}

\begin{remark} Notice that these equations generalize the ones introduced in \cite{LGMRR_23}. Indeed, denoting by $\Theta_h = \Theta + h$ the equations now read as 
\[
\psi^\ast \Theta = 0  \quad \text{and}  \quad \psi^\ast \iota_\xi\overline{{\d}} \Theta_h = 0\,, \quad  \forall \xi \in \X(M)
\]
where $\overline{\d} = \d + \sigma_h\wedge$. A local computation in the case of fiber bundles $\pi \colon Y \rightarrow X$ may be found in Section \ref{section:Dissipative_field_theories}.
\end{remark}

In order to discuss dissipation phenomena (dissipated forms), it will be useful for us to have an explicit formula for the evolution of a conformal Hamiltonian $(n-1)$-form $\alpha \in \Omega^{n-1}_H(M)$ in terms of the Hamiltonian $h$ and the dissipation $1$-form $\sigma_h$. This is given by the following:

\begin{theorem}
\label{thm:Evolution_of_Hamiltonian_forms}
Let $h \in \Gamma(\mathcal{H})$ be a good Hamiltonian, let $\alpha \in \Omega^{n-1}_H(M)$ be a conformal Hamiltonian form and let $\psi\colon X \rightarrow M$ be a solution of the Hamilton--de Donder--Weyl equations defined by $h$. Then
\[
\psi^\ast(\d \alpha) = \psi^\ast \left( - \mathcal{R}(\d \alpha) h - \iota_{X_\alpha} \d h - \sigma_h \wedge \alpha + \sigma_h \wedge \iota_{X_\alpha} h\right)\,.
\]
\end{theorem}
\begin{proof} Indeed, we know that $\psi^{\ast} \iota_{\xi} (\d + \sigma_h \wedge)(\Theta + h) =  0$, for every $\xi \in \X(M)$. Let $\xi = X_\alpha$, where $X_\alpha$ is a vector field satisfying 
\[
\iota_{X_\alpha} \Theta = - \alpha\qquad\text{and}\qquad  \iota_{X_\alpha} \d \Theta = \d \alpha - \mathcal{R}(\d \alpha) \Theta\,.
\]
Then, we obtain
\begin{align*}
    \psi^\ast\left( \d \alpha - \mathcal{R}(\d \alpha) \Theta + \iota_{X_\alpha} \d h +\sigma_h(X_\alpha) \Theta + \sigma_h \wedge \alpha + \sigma_h (X_\alpha) h - \sigma_h \wedge \iota_{X_\alpha} h \right) = 0
\end{align*}
which, using that $\psi^\ast \Theta = - \psi^\ast h$, implies
\[
\psi^\ast(\d \alpha + \mathcal{R}(\d \alpha) h + \iota_{X_\alpha} \d h + \sigma_h \wedge \alpha - \sigma_h \wedge \iota_{X_\alpha} h) = 0\,,
\]
finishing the proof.
\end{proof}

\begin{corollary} 
\label{cor:Evolutio_of_vertical_forms}
Under the hypotheses of Theorem \ref{thm:Evolution_of_Hamiltonian_forms}, if $X_\alpha$ takes values in $(\Ima \flat_\Theta)^{\circ, 1}$, we have 
\[
\psi^\ast(\d \alpha) = \psi^\ast \left( - \mathcal{R}(\d \alpha) h - \iota_{X_\alpha} \d h - \sigma_h \wedge \alpha \right)\,.
\]
\end{corollary}
\begin{proof} Clearly, since $h$ takes values in $\mathcal{H}$, which is defined as the subbundle of forms that vanish when contracted by an element of $(\Ima \flat_\Theta)^{\circ, 1}$, so that $\iota_{X_\alpha} h = 0$.
\end{proof}

A natural question to ask is what are the conditions that guarantee that every possible choice of Hamiltonian is a good Hamiltonian. Consider $h \in \Gamma(\mathcal{H})$, $R \in \ker_1\d \Theta$ and $v \in \T M$. Abusing of notation, let us denote by $R$ and $v$ extensions to global vector fields. Then, $\iota_R \d h $ takes values in $\mathcal{H}$ if and only if $\iota_v \iota_R \d h$ takes values in $\Ima \flat_\Theta$. We have the following
\begin{align*}
\iota_v \iota_R \d h = \iota_v \Lie_R h - \iota_v \d \iota_R h = \Lie_R\iota_vh - \iota_{[R, v]} h - \iota_v \d \iota_R h\,.
\end{align*}
Clearly, by definition, $\iota_{[R,v]} h$ takes values in $\Ima \flat_\Theta$. This implies that $\Lie_R\iota_vh- \iota_v \d \iota_R h$ measures the extent to which $h$ fails to be a good Hamiltonian. Define
\[
\begin{array}{rccc}
\gamma_\Theta\colon & \Gamma(\mathcal{H}) \times \Gamma(\ker_1 \d \Theta) \times \X(M) & \longrightarrow & \Omega^{n-1}(M)\big / \Gamma(\Ima \flat_\Theta)\\
& (h , R , v) & \longmapsto & \Lie_R\iota_vh- \iota_v \d \iota_R h\,.
\end{array}
\]
This map is $\Cinfty(M)$-linear in both $R$ and $v$ but, unfortunately, fails to be $\Cinfty(M)$-linear in $h$, in general. Therefore, it induces a map (which, abusing of notation, is denoted by $\gamma_\Theta$ as well)
\[\gamma_\Theta \colon \Gamma(\mathcal{H}) \times \left( \ker_1 \d \Theta \otimes \T M\right) \longrightarrow \bigwedge^{n-1}M \Big/ \Ima \flat_\Theta\,,\]
so that we have the following.
\begin{corollary} Let $(M, \Theta)$ be a multicontact manifold. Then, every $h \in \Gamma(\mathcal{H})$ is a good Hamiltonian if and only if $\gamma_\Theta$ vanishes identically.
\end{corollary}

Although $\gamma_\Theta$ can be difficult to compute in complete generality, if we restrict the study to a particular subclass of multicontact manifolds, $\gamma_\Theta$ induces a tensor which is much easier to manage.

\begin{definition} A multicontact manifold $(M, \Theta)$ is called \textbf{variational} if
$$ \iota_{\bigwedge^2 (\ker_1 \d \Theta)} \Theta = 0\,. $$
\end{definition}
\begin{proposition}
\label{prop:Reeb_contracted_with_Hamiltonian}
Let $(M, \Theta)$ be a variational multicontact manifold and $h \in \Gamma(\mathcal{H)}$. Then, $\iota_{\ker_1 \d \Theta} h = 0$. 
\end{proposition}

\begin{proof} Indeed, let $v \in \T M$ and $R \in \ker_1 \d \Theta$. Then, 
\[
\iota_v \iota_R h = - \iota_R \iota_v h = - \iota_R \iota_{R'} \Theta\,,
\]
for certain $R' \in \ker_1{\d \Theta}$. Now, since $\Theta$ is variational, $\iota_{R \wedge R'} \Theta = 0$, which finishes the proof.
\end{proof}
Proposition \ref{prop:Reeb_contracted_with_Hamiltonian} implies that, for variational multicontact manifolds, we may write \[
\gamma_\Theta (h, R, v) = \Lie_R \iota_v h = \iota_R \d \iota_v h\,
\]
and, now, it is $\Cinfty(M)$-linear on every component. Indeed, it suffices to observe that 
\begin{align*}
    \gamma_\Theta(f h, R, v) = \gamma_\Theta(h, R, fv) &=f\gamma_\Theta(h, R, v) + R(f) \iota_v h \quad\operatorname{mod} \Ima \flat_\Theta\\
    &= f \gamma_\Theta(h, R, v)\operatorname{mod} \Ima \flat_\Theta\,,
\end{align*}
where in the last equality we have used $\iota_v h \in \Ima \flat_\Theta$ (see Definition \ref{def:Hamiltonian}). Hence, it induces a vector bundle map
\[\gamma_\Theta \colon \mathcal{H} \otimes \ker_1 \d \Theta \otimes \T M \longrightarrow \bigwedge^{n-1} M \Big / \Ima \flat_\Theta\,.\] 

However, we can obtain a much simpler tensor that also measures the extent to which Hamiltonians fail to be good Hamiltonians. Indeed, notice that $\iota_v h = \iota_{R'} \Theta$, for some $R' \in \ker_1 \d \Theta$ and we may write $\gamma_\Theta(h, R, v) = \iota_R \d \iota_{R'} \Theta$. This induces the following definition
\begin{definition} Let $(M, \Theta)$ be a variational multicontact manifold. Then, its {\bf distortion} is defined as the following map
\[
\begin{array}{rccc}
C_\Theta \colon & \Gamma(\ker_1 \d \Theta) \otimes \Gamma(\ker_1 \d\Theta) & \longrightarrow & \Omega^{n-1}(M) / \Gamma(\Ima \flat_\Theta)\\
& R\otimes R' & \longmapsto & \iota_{R}\d \iota_{R'} \Theta\,.
\end{array}
\]
\end{definition}

\begin{proposition}For a variational multicontact manifold, $C_\Theta$ defines a symmetric vector bundle map 
\[C_\Theta\colon \ker_1 \d \Theta \otimes \ker_1 \d \Theta \longrightarrow{\bigwedge^{n-1} M\big / \Ima \flat_\Theta}\,.\]
\end{proposition}
\begin{proof} It is clearly $\Cinfty(M)$-linear on its first component. Therefore, it suffices to show that it is symmetric to conclude the result. Indeed,
\begin{align*}
    \iota_R \d \iota_{R'} \Theta &= \iota_{R}\Lie_{R'} \Theta - \iota_{R} \iota_{R'} \d \Theta =\iota_{R}\Lie_{R'} \Theta\\
    &= \Lie_{R'} \iota_{R} \Theta - \iota_{[R', R]} \Theta\\
    &= \iota_{R'} \d \iota_{R} \Theta  - \iota_{[R', R]} \Theta\,.
\end{align*}
Notice that $[R', R]$ takes values in $\ker_1 \d \Theta$, and so $\iota_{[R', R]} \Theta$ takes values in $\Ima \flat_\Theta$, concluding that $C_\Theta$ is symmetric and finishing the proof.
\end{proof}
Concluding, we have:
\begin{theorem}
\label{thm:distortion}
Let $(M, \Theta)$ be a  variational multicontact manifold such that $\ker_1 \d \Theta$ has constant rank. Then, every Hamiltonian is a good Hamiltonian if  and only if its distortion vanishes.
\end{theorem}

\subsection{Dissipated forms}

In contact dynamics, a fundamental concept is that of \textit{dissipated quantity} \cite{GGMRR_20a}, which is a function $g \in \Cinfty(M)$ satisfying $\dot{g} = X_H(g) = - R(H) g$. Recently, in \cite{symm}, the concept was generalized to the multicontact formulation of field theories.

\begin{definition}
\label{def:dissipated_form}
Let $(M, \Theta)$ be a multicontact manifold and let $h \in \Gamma(\mathcal{H})$ be a good Hamiltonian. Then, a form $\alpha \in \Omega^{a}_H(M)$ is said to be \textbf{dissipated}, if 
\[\psi^\ast( \d \alpha) = - \sigma_h \wedge \alpha\,,\]
for every $\psi$ solution of the multicontact Hamilton--de Donder--Weyl equations determined by $h$.
\end{definition}

\begin{remark} Notice that Definition \ref{def:dissipated_form} recovers the usual notion of dissipated quantity of contact mechanics (see \cite{LL_19}). Indeed, in the context of contact mechanics, the dissipation $1$-form reads as $\pdv{H}{z}\d t$ (see \cite{LGMRR_23, LGMRR_25}) so that the condition on a function $f \in \Cinfty(M)$ to be dissipated is for a curve $\gamma \colon \mathbb{R} \rightarrow M$ to satisfy
\[
\psi^\ast (\d f) = -\pdv{H}{z} f \d t\,,
\]
or equivalently, $\dot{f} = -\pdv{H}{z} f$.
\end{remark}

\begin{remark}
\label{remark:Comparison_dissipated_forms}
The definition of dissipated form that can be found in \cite{symm} is not quite the same as Definition \ref{def:dissipated_form}. Indeed, in the cited article, a differential form $\alpha \in \Omega^a(M)$ is called dissipated if $\psi^\ast \left(  \overline{\d} \alpha \right)= 0$, for every map $\psi\colon X \rightarrow M$ satisfying $\psi^\ast \Theta = 0$ and $\psi^\ast \iota_\xi \overline{\d} \Theta = 0$, where $\xi \in \mathfrak{X}(M)$ is an arbitrary vector field and $\overline{\d} \alpha = \d \alpha+\sigma\wedge \alpha $, $\sigma$ denoting a particular $1$-form, which only depends on the multicontact form $\Theta$. This last form is also called the dissipation form, but does not completely correspond to the notion introduced in this text (see \cite{LGMRR_23, LGMRR_25}).

In our setting, the dissipation form depends upon a choice of Hamiltonian, $h$, and this choice induces an operator $\overline{\d}_h \alpha = \d \alpha + \sigma_h\wedge \alpha$, which allows us to write the condition for a form to be dissipated as $\psi^\ast \left( \overline{\d}_h \alpha \right) = 0$, for every solution of Hamilton--de Donder--Weyl equations $\psi$ or, equivalently, for any smooth map $\psi\colon X \rightarrow M$ satisfying $\psi^\ast \left(\Theta +h \right) = 0$ and $\psi^\ast \iota_\xi\overline{\d}_h \left( \Theta + h\right) = 0$. Hence, we recover the definition presented in \cite{symm} if we work with the form $\Theta_h = \Theta + h$, instead of interpreting $\Theta$ to fix geometry and $h$ to define the dynamics. It can then be shown that $\sigma_h$ is the dissipation form associated to $\Theta_h$ in the sense of \cite{LGMRR_23, LGMRR_25} (see Lemma \ref{lemma:formula_for_dissipation_form}). 
\end{remark}

In the case of Hamiltonian forms, we can obtain a condition to identify dissipated forms, which will be useful in the sequel to find a subalgebra of dissipated forms. Indeed, using Theorem \ref{thm:Evolution_of_Hamiltonian_forms} we have the following

\begin{proposition}
\label{prop:Dissipative_condition} Let $\alpha \in \Omega^{n-1}_H(M)$. Then, if 
\[- \left( \Lie_{X_\alpha} + \mathcal{R}(\d \alpha)\right) h + \left( \d + \sigma_h \wedge\right)\iota_{X_\alpha} h = 0\,,\]
then $\alpha$ is a dissipated $(n-1)$-form.
\end{proposition}

\begin{remark}Notice that if $X_\alpha$ takes values in $(\Ima \flat_\Theta)^{\circ, 1}$, then $\iota_{X_\alpha} h = 0$ and therefore, the equation in Proposition \ref{prop:Dissipative_condition} simplifies to $- \left( \Lie_{X_\alpha} + \mathcal{R}(\d \alpha)\right) h = 0$, which in the case of mechanics simplifies to $ - (X_f(H) + \pdv{f}{z} H) = \{f, H\} = 0$.
\end{remark}

A natural question to ask is if dissipated forms are closed under Jacobi bracket. As we know the vector field associated to $\{\alpha, \beta\}$ is $[X_\alpha, X_\beta]$. Also, a quick computation shows $\mathcal{R}(\d\{\alpha, \beta\}) = X_\alpha(\mathcal{R}(\d \beta)) - X_\beta(\mathcal{R}(\d \alpha))$. Therefore, after a simple, but rather lengthy computation, we get
\begin{align*}
   &- \left(\Lie_{[X_\alpha, X_\beta]} + \mathcal{R}(\d \{\alpha, \beta\})\right) h + (\d +\sigma_h \wedge) \iota_{[X_\alpha, X_\beta]} h =\\
   &\qquad\qquad\qquad\qquad\qquad\qquad\qquad = \left(\Lie_{X_\alpha} + \mathcal{R}(\d \alpha) \right) \left( - \left( \Lie_{X_\beta} + \mathcal{R}(\d \beta)\right) h + \left( \d + \sigma_h \wedge\right)\iota_{X_\beta} h\right)\\
   &\qquad\qquad\qquad\qquad\qquad\qquad\qquad\qquad - (\Lie_{X_\beta} + \mathcal{R}(\d \beta) )\left(- \left( \Lie_{X_\alpha} + \mathcal{R}(\d \alpha)\right) h + \left( \d + \sigma_h \wedge\right)\iota_{X_\alpha} h  \right)\\
   &\qquad\qquad\qquad\qquad\qquad\qquad\qquad\qquad + (\d + \sigma_h) \Lie_{X_\alpha \wedge X_\beta} h + \Lie_{X_\beta} \sigma_h \wedge \iota_{X_\alpha} h - \Lie_{X_\alpha} \sigma_h \wedge \iota_{X_\beta} h \\
   &\qquad\qquad\qquad\qquad\qquad\qquad\qquad\qquad + \mathcal{R}(\d \beta) (\d + \sigma_h) \iota_{X_\alpha} h - \mathcal{R}(\d \alpha)(\d + \sigma_h) \iota_{X_\beta} h\,.
\end{align*}
As a corollary, we obtain the following:
\begin{corollary} Let $\alpha$ and $\beta$ be conformal Hamiltonian $(n-1)$-forms such that 
\[- \left( \Lie_{X_\alpha} + \mathcal{R}(\d \alpha)\right) h = 0 \qquad\text{and}\qquad  -\left( \Lie_{X_\beta} + \mathcal{R}(\d \beta)\right) h = 0\,.\]
If both $X_\alpha$ and $X_\beta$ take values in $(\Ima \flat_\Theta)^{\circ, 1}$, then $\{\alpha, \beta\}$ is a dissipated form.
\end{corollary}
\begin{proof} It is enough to notice that all terms vanish from the last expression and use Proposition \ref{prop:Dissipative_condition}.
\end{proof}

\subsection{Multisymplectization of the dynamics}

Given a contact manifold, say $(M, \eta)$, and a Hamiltonian $H \in \Cinfty(M)$, we can define the homogenous Hamiltonian on its canonical symplectization $(\widetilde M:= M \times \mathbb{R}_\times, - \d (z \,\eta))$ as $\widetilde H(p, z) := z H(p)$. Then, the symplectic dynamics of $\widetilde H$ defined by the Hamiltonian vector field $X_{\widetilde H} \in \X(\widetilde M)$ are $\tau$-related with the contact dynamics defined by the (contact) Hamiltonian vector field $X_{H} \in X(M)$. Furthermore, the evolution of observables are conveniently related (see \cite{GG_22}). We now study this relationship in the general setting presented in this paper.

First, let us recall the notion of multisymplectic Hamilton--De Donder--Weyl equations.

\begin{definition}
\label{def:Multisymplectic_HDW_equations}
Let $(\widetilde M, \Omega)$ be a multisymplectic manifold and $\widetilde h \in \Omega^{n}(\widetilde M)$ be an arbitrary form (which we refer to as the Hamiltonian). Then, a map $\widetilde \psi \colon X \longrightarrow \widetilde M$ from an $n$-dimensional manifold $X$ to $\widetilde M$, is said to satisfy {\bf the multisymplectic Hamilton--De Donder--Weyl equations} if 
\[
\widetilde \psi^\ast \iota_\xi\left( \Omega - \widetilde h\right) = 0\,, \quad \forall \xi \in \X(\widetilde M)\,.
\]
\end{definition}

Now, we may prove the following characterization of wether solutions may be lifted or projected under certain assumption on the maps $\psi \colon X \rightarrow M$, related to the transversality of the maps to be considered. For instance, when the multicontact structure comes from a fibered manifold $\pi \colon Y \rightarrow X$ (see Section \ref{section:Field_equations}), this is immediately ensured.

\begin{theorem}
\label{thm:Liftting_HDW_solutions}
Let $(M, \Theta)$ be a variational multicontact manifold and let $(\widetilde M, \Omega)$ denote its canonical multisymplectization. Given a good Hamiltonian $h \in \Omega^n(M)$, define the associated homogeneous Hamiltonian as $\widetilde h(p, z):= z \tau^\ast h(p)$. Then, we have the following:

\begin{enumerate}[\rm (i)]
    \item Let $\widetilde \psi \colon X \longrightarrow \widetilde M$ be a solution to the multisymplectic Hamilton--De Donder--Weyl equations defined by $\widetilde h$ such that, if we denote $\psi:= \tau \circ \widetilde \psi$, we have
    \[
    \bigwedge ^{n-1} X = \langle \widetilde \psi^\ast \left(\iota_\xi \Theta \right) \colon \xi \in \ker \d \Theta \rangle \qquad \text{and} \qquad \widetilde{\mathcal{R}}|_{\psi(X)} = \xi^i \otimes \left(\psi_\ast(U_i) + (\Ima \flat_\Theta)^{\circ, n-1} \right)\,,
    \]
    for certain $\xi^i \in \ker_1 \d \Theta$ and $U_i \in \X^{n-1}(X)$. Then $\psi := \tau \circ \widetilde \psi \colon X \longrightarrow M$ is a solution of the multicontact Hamilton--De Donder--Weyl equations defined by $h$.
    \item Conversely, let $\psi \colon X \rightarrow M$ be a solution of the multicontact Hamilton--De Donder--Weyl equations defined by $h$ such that 
    \[
    \bigwedge ^{n-1} X = \langle  \psi^\ast \left(\iota_\xi \Theta \right) \colon \xi \in \ker \d \Theta \rangle\qquad \text{and} \qquad \widetilde{\mathcal{R}}|_{\psi(X)} = \xi^i \otimes \left(\psi_\ast(U_i) + (\Ima \flat_\Theta)^{\circ, n-1} \right)\,,
    \]
    for certain $\xi^i \in \ker_1 \d \Theta$ and $U_i \in \X^{n-1}(M)$. Then there is a lift 
\[\begin{tikzcd}[cramped]
	& {\widetilde M} \\
	X & M
	\arrow["\tau"{description}, from=1-2, to=2-2]
	\arrow["{\widetilde \psi}"{description}, from=2-1, to=1-2]
	\arrow["\psi"{description}, from=2-1, to=2-2]
\end{tikzcd}\,,\]
    satisfying the multisymplectic Hamilton--De Donder--Weyl equations for $\widetilde h$ if and only if there is a function $g\in \Cinfty(X)$ such that 
    \[
    \psi^\ast \sigma_h = \d g\,.
    \]
    In this case, the lift can be written as $\widetilde \psi = (\psi, e^g)$.
    \item Finally, under the same hypotheses, there is a correspondence between the evolution of observables, namely for $\alpha \in \Omega^{a}_H(M)$ conformal Hamiltonian, defining $\widetilde \alpha (p, z) := z \, \tau^\ast \alpha(p)$, we have 
    \[
    \widetilde \psi^\ast (\d \widetilde \alpha) = (\widetilde \psi^\ast z) \cdot \, \widetilde \psi^\ast \left( \sigma_h \wedge \alpha + \d \alpha \right)\,. 
    \]
\end{enumerate}
\end{theorem}

\begin{remark} Notice that Theorem \ref{thm:Liftting_HDW_solutions} implies that there is a correspondence between closed forms on solutions in the homogeneous multisymplectic setting and dissipated forms in the multicontact setting.
\end{remark}

Before proving the Theorem, let us first prove an useful formula:

\begin{lemma}
\label{lemma:formula_for_dissipation_form}
Let $(M, \Theta)$ be a variational multicontact manifold and let $\psi \colon X \longrightarrow M$
be a map from an $n$-dimensional manifold $X$, such that there exists $U_i \in \X^{n-1}(X)$ and $\xi^i \in \X(M)$ taking values in $\ker \d \Theta$ for which 
\[
\widetilde{\mathcal{R}}|_{\psi(X)} = \xi^i\otimes \left( \psi_\ast (U_i) + \left(\Ima \flat_\Theta\right)^{\circ, n-1}\right)
\]
Then, for every $\xi \in \ker  \d \Theta$ and every good Hamiltonian $h \in \Omega^n(M)$, we have 
    \[
    \psi^\ast \left(\sigma_h \wedge \iota_\xi \Theta \right) = \psi^\ast(\iota_\xi \d h).
    \]
\end{lemma}
\begin{proof} 
Let $\xi^i$ be a basis for $\ker_1 \d \Theta$, let $U_i$ be multivector fields as in the hypothesis and let $\xi = a_i \xi^i \in \ker \d \Theta$. Then, 
Indeed, we have 
\begin{align*}
    \psi^\ast \left (\sigma_h \wedge \iota_\xi \Theta \right) &=(-1)^{n-1} n \cdot \psi^\ast \left( \iota_{\widetilde{\mathcal{R}}} \d h\right) \wedge \psi^\ast \left( \iota_\xi \Theta\right) = (-1)^{n-1}n \left(\iota_{\psi_\ast(U_i)} \psi^\ast\left( \iota_{\xi^i} \d h\right)\right) \wedge \psi^\ast \left(\iota_\xi \Theta \right)\\
    &= n\cdot \psi^\ast \left( \iota_{\xi^i} \d h\right) \wedge \psi^\ast \left( \iota_{\psi_\ast(U_i)} \iota_{\xi} \Theta\right)\\
    &= \psi^\ast(\iota_{\xi^i} \d h) \cdot a_i = \psi^\ast(\iota_\xi \d h)\,,
\end{align*}
which finishes the proof.
\end{proof}

\begin{remark} Lemma \ref{lemma:formula_for_dissipation_form} relates the definition for the dissipation $1$-form given in this paper with the one presented in \cite{LGMRR_23}. Indeed, it is a generalization.
\end{remark}

\begin{proof}[Proof of Theorem \ref{thm:Liftting_HDW_solutions}]
    \begin{enumerate}[\rm (i)]
        \item Taking $\xi = \pdv{z}$ on $\widetilde M$, we verify immediately that $\widetilde \psi$ satisfies $\widetilde \psi^\ast \tau^\ast \left(\Theta + h\right) = 0$, given that $\widetilde \psi$ satisfies the multisymplectic Hamilton--De Donder--Weyl equations, so that $\psi^\ast \left( \Theta + h\right) = 0$. Now, let us write
        \begin{align}
            0 &= \widetilde \psi^\ast \iota_\xi\Omega = \widetilde\psi^\ast \left( \d z \wedge \iota_{\xi} \left( \Theta + h\right) - \d z (\xi)(\Theta + h) -z \iota_\xi(\Theta + h) \right)\\
            &=\widetilde \psi^\ast \left( \d z \wedge \iota_\xi \left(\Theta + h\right) - z \iota_\xi \left( \d \Theta + \d h\right)\right). \label{eq:multicontact_equations}
        \end{align}
    In particular, taking $\xi \in \ker \d \Theta$, we find, in light of Proposition \ref{prop:Reeb_contracted_with_Hamiltonian} and the definition of Hamiltonian, that 
    \[
    \widetilde \psi^\ast \left( \d z \wedge \iota_\xi \Theta \right) = \widetilde \psi^\ast \left(z\, \iota_\xi \d h \right)\,.
    \]
    Letting $f := \widetilde \psi^\ast z$, we may write the above equality as $\frac{\d f}{f} \wedge \psi^\ast \left( \iota_\xi \Theta\right) = \psi^\ast (\iota_\xi \d h )$. We are now going to show that \[
    \psi^\ast \sigma_h = \frac{\d f}{f}\,,
    \]
    which will conclude the proof, after substituting in Eq. \eqref{eq:multicontact_equations}. Indeed, by Lemma \ref{lemma:formula_for_dissipation_form}, we have 
    \[
    \left(\psi^\ast\sigma_h - \frac{\d f}{f} \right) \wedge \psi^\ast \iota_\xi \Theta = 0\,,
    \]
    for every $\xi \in \ker \d \Theta$. Since, by hypotheses, $\langle\psi^\ast \iota_\xi \Theta \colon \xi \in \ker \d \Theta\rangle = \bigwedge^{n-1} X$, we have 
    \[
    \left(\psi^\ast \sigma_h - \frac{\d f }{f}\right) \wedge \alpha = 0\,,
    \]
    for every $\alpha \in \bigwedge^{n-1}X$, which implies the desired equality, finishing the proof.
    \item It is straight-forward to prove the converse, following the same reasoning as in {\rm (i)}.
    \item It follows in a straight-forward manner, using the fact that $\widetilde \psi^\ast (\d z ) = \widetilde \psi^\ast(z) \cdot \psi^\ast \sigma_h$.
    
    \end{enumerate}
\end{proof} 

\begin{remark} Notice that Theorem \ref{thm:Liftting_HDW_solutions} {\rm (iii)} implies that there is a correspondence between conserved forms (forms that are closed on solutions) in the multisymplectic setting and disipated forms in the multicontact setting, provided that $\psi^\ast \sigma_h$ is exact.
\end{remark}

\begin{remark}The notion of lifting solutions to the canonical multisymplectization may be generalized to an arbitrary multisymplectization $\tau \colon \widetilde M \longrightarrow M$ with nowhere vanishing conformal factor $\phi \in \Cinfty(\widetilde M)$ by introducing the homogenous Hamiltonian $\widetilde h := \phi \cdot \tau^\ast h$. Under the same technical requirement, statement {\rm (i)} of Theorem \ref{thm:Liftting_HDW_solutions} follows in a straightforward manner and so does {\rm (iii)}, using Theorem \ref{thm:Classification_multisymplectization}. Of course, in full generality, {\rm (ii)} fails to be true since there may not even be a lift $\widetilde \psi \colon X \longrightarrow \widetilde M$.
\end{remark}




\section{Dissipative field theories}
\label{section:Dissipative_field_theories}
In this section we apply the  constructions developed along this paper to the case of dissipative field theories.

Let $\pi\colon Y \rightarrow X$ be a fiber bundle (or, more generally, a fibration). The extended phase space of action dependent field theories is defined as (see \cite{LGMRR_23,LGMRR_25,GLMR_24})
\[
M := \bigwedge^n_2 Y \oplus \bigwedge^{n-1} X\,,
\]
where $\bigwedge^n_2 Y$ denotes the bundle of $2$-horizontal $n$-forms on $Y$ and $\bigwedge^{n-1} X$ denotes the bundle of $(n-1)$-forms on $X$. We can find a canonical multicontact form on $M$ given by 
\[\Theta = \d \Theta_X^{n-1} - \Theta_Y ^{n}\,,\]
where $\Theta_Y^{n}$ and $\Theta^{n-1}_X$ denote the canonical Liouville forms on $\bigwedge^n_2 Y$ and $\bigwedge^{n-1} X$, respectively. If $(x^\mu, y^i)$ denote fibered coordinates on $\pi\colon Y \rightarrow X$, then we have canonical coordinates $(x^\mu, y^i, p, p^\mu_i, s^\mu)$ on $M$. Indeed, any element in $\bigwedge^n_2 Y$ (respectively, in $\bigwedge^{n-1} X$) can be locally written as

\[p \d^n x + p^\mu_i \d y^i \wedge \d^{n-1}x_\mu\qquad  (\,\text{respectively} \; s^\mu \d ^{n-1}x_\mu\,)\,,
\]
where $\d^nx = \d x^1 \wedge \cdots \wedge \d x^n$ and $\d^{n-1}x_\mu = i_{\frac{\partial}{\partial x^\mu} }\d^nx$.

The coordinate expression of the canonical multicontact form and its exterior differential under this choice read
\[\Theta = \d s^\mu \wedge \d^{n-1}x_\mu -  p \d^n x - p^\mu_i \d y^i \wedge \d^{n-1}x_\mu\,, \qquad \d \Theta =  -  \d p \wedge \d^n x - \d p^\mu_i \wedge d y^i \wedge \d^{n-1}x_\mu\,.\]
Therefore, we have 
\[\ker_1 \Theta = \left \langle\pdv{y^i} - p^\mu_i \pdv{s^\mu}\,, \pdv{p^\mu_i}\right\rangle \qquad\text{and}\qquad \ker_1 \d \Theta = \left\langle  \pdv{s^\mu}\right\rangle\,,\]
which implies that $\Theta$ defines a multicontact structure.

As we showed in Section \ref{section:Field_equations}, the best behaved conformal Hamiltonian forms are those whose infinitesimal conformal transformation takes values in $(\Ima \flat_{\Theta})^{\circ, 1}$. In our case, 
\[\flat_{\theta} \colon \ker_1 \d \Theta \longrightarrow \bigwedge^{n-1} M\]
and has image $\langle \d^{n-1}x_\mu\rangle$, so that $\Ima \flat_\Theta$ is the space of horizontal forms with respect to the fibered structure $M \rightarrow X$. In particular, $(\Ima \flat)^{\circ, 1}$ is the vertical distribution. Let us study those Hamiltonian forms under the previous multicontact structure with vertical infinitesimal conformal transformation with respect to the projection $M \rightarrow X$.

\begin{proposition}
\label{prop:Conf_transformation_locally}
Let $n > 1$ and $X \in \X(M)$ be a vertical infinitesimal conformal transformation. Then, its local expression is
\begin{multline}
    X = \left( F s^\mu + G^\mu - p^\mu_j \pdv{G^\nu}{p^\nu_j}\right) \pdv{s^\mu} - \pdv{G^\mu}{p^\mu_i} \pdv{y^i} \\
    + \left( \pdv{F}{y^i} s^\mu + \pdv{G^\mu}{y^i} + F p^\mu_i \right) \pdv{p^\mu_i} + \left( \pdv{F}{x^\mu} s^\mu + \pdv{G^\mu}{x^\mu} + F p \right) \pdv{p}\,,
\end{multline}
where $F$ is independent of $s^\mu$, $p$ and $p^\mu_i$, and $G^\mu$ satisfies $\dparder{G^\mu}{p^\nu_i} = \delta^{\mu}_\nu \dparder{G^\alpha}{p^\alpha_i}$.
\end{proposition}

\begin{proof} Let \[
X = A^\mu \pdv{s^\mu} + B^i \pdv{y^i} + C^\mu_i \pdv{p^\mu_i} + D \pdv{p}\,.
\]
Then,
\begin{align*}
    \d \iota_{X} \Theta &= \d \left(A^\mu \d^{n-1}x_\mu - p^\mu_i B^i \d^{n-1}x_{\mu} \right)\\
    &= \left(\pdv{A^\mu}{x^\mu }-p^i_\mu \pdv{B^i}{x^\mu} \right)\d^n x  + \pdv{y^i}\left(A^\mu - p^\mu_j B^j\right)\d y^i \wedge \d^{n-1}x_\mu\\
    &\quad + \pdv{p^\nu_i}\left(A^\mu - p^\mu_j B^j\right)\d p^\nu_i \wedge \d^{n-1}x_\mu + \pdv{p}\left(A^\mu - p^\mu_j B^j\right)\d p \wedge \d^{n-1}x_\mu\\
    &\quad + \pdv{s^\nu}\left(A^\mu - p^\mu_j B^j\right)\d s^\nu \wedge \d^{n-1}x_\mu\,,
\end{align*}
and
\[
    \iota_{X} \d \Theta = - D \d^n x - C^\mu_i \d y^i \wedge \d ^{n-1}x_\mu + B^i \d p^\mu_i \wedge \d^{n-1}x_\mu\,,
\]
so that, if $X$ defines an infinitesimal conformal transformation, $\Lie_X \Theta = f \cdot \Theta$ for certain $f \in \Cinfty(M)$ and we must have
\begin{align*}
  &\left( \pdv{A^\mu}{x^\mu } -p^i_\mu\pdv{B^i}{x^\mu} - D\right)\d^n x + \left(\pdv{y^i}\left(A^\mu - p^\mu_j B^j\right) - C^\mu_i\right) \d y^i \wedge \d^{n-1}x_\mu\\
  &\quad + \left( \pdv{p^\nu_i}\left(A^\mu - p^\mu_j B^j\right) + \delta^\mu_\nu B^i\right) \d p^\nu_i \wedge \d^{n-1}x_\mu+ \pdv{p}\left(A^\mu - p^\mu_j B^j\right)\d p \wedge \d^{n-1}x_\mu\\
  &\quad + \pdv{s^\nu}\left(A^\mu - p^\mu_j B^j\right)\d s^\nu \wedge \d^{n-1}x_\mu = \\
  &\qquad\qquad\qquad\qquad\qquad\qquad = f \left( \d s^\mu \wedge \d^{n-1}x_\mu - p \d^n x - p^\mu_i \d y^i \wedge \d^{n-1}x_\mu \right)\,.
\end{align*}
Contracting with $\dfrac{1}{n} \dparder{}{s^\mu} \wedge \dparder{^{n-1}}{^{n-1}x_\mu}$ both sides of the equality, we obtain
\[f = \frac{1}{n}\pdv{s^\mu} \left( A^\mu - p^\mu_j B^j\right)\,.\]
Comparing the terms with $\d s^\nu$, we conclude that
$A^\mu - p^\mu_j B^j = Fs^\mu + G^\mu$, where $F$ and $G^\mu$ are independent of $s^\nu$. In particular,
 $C^\mu_i = \pdv{F}{y^i} s^\mu + \pdv{G^\mu}{y^i} + p^\mu_i F$ and $D = \pdv{F}{x^\mu} s^\mu + \pdv{G^\mu}{x^\mu} + p F$. Since the term with $\d p$ must vanish, both $F$ and $G^\mu$ are independent of $p$ and since the term with $\d p^\nu_i$ must vanish we arrive at the condition
 \[\pdv{F}{p^\nu_i} s^\mu + \pdv{G^\mu}{p^\nu_i} = - \delta^\mu_\nu B^i\,.\]
Taking $\nu \neq \mu$, we conclude that $F$ is independent of $p^\mu_i$ so we are left with the equation $\pdv{G^\mu}{p^\nu_i} = - \delta^\mu_\nu B^i$, proving the result.
\end{proof}

Then, using Proposition \ref{prop:Conf_transformation_locally}, we conclude the following

\begin{corollary} A Hamiltonian form $\alpha \in \Omega^{n-1}_H(M)$ with vertical infinitesimal conformal transformation has the following local expression:
\[\alpha = (- F s^\mu - G^\mu) \d^{n-1} x_\mu\,,\]
where $F = F(x, y)$, and $G^\mu$ satisfies $\pdv{G^\mu}{p^\nu_i} = \delta^{\mu}_\nu \pdv{G^\alpha}{p^\alpha_i}$.
\end{corollary}

In Table \ref{table:examples}, the reader can find some elementary examples of conformal Hamiltonian forms, their associated infinitesimal conformal transformations, and conformal factors. The Jacobi brackets of each possible pair can be found in Table \ref{table:brackets}.

\begin{table}[ht]
\caption{Elementary examples of conformal Hamiltonian forms}
\centering
\renewcommand{\arraystretch}{2}
  \begin{tabular}{c|cc}
    Form & Infinitesimal conformal transformation & Conformal factor \\
    \hline
    $s^\mu \d ^{n-1}x_\mu$ & $-s^\mu \pdv{s^\mu} -p^\mu_i \pdv{p^\mu_i} -p \pdv{p}$ & $-1$ \\
     \hline
    $y^i \d^{n-1}x_\mu$ & $-y^i \pdv{s^\mu}  -\pdv{p^\mu_i}$ & 0 \\
     \hline
     $p^\mu_i \d^{n-1}x_\mu$ & $ \pdv{y^i}$ & 0 \\
     \hline
     $\d^{n-1}x_\mu$ & $-\pdv{s^\mu}$ & $0$
  \end{tabular}
  \label{table:examples}
\end{table}

\renewcommand{\arraystretch}{2}

\begin{table}[ht]
\centering
\caption{Jacobi brackets of the elementary examples}
    \begin{tabular}{c|cccc}
$\{\cdot,\cdot\}$ & $s^\nu \d^{n-1}x_\nu$ & $y^j \d^{n-1}x_\nu$ & $p^\nu_j \d^{n-1}x_\nu$ & $\d^{n-1}x_\nu$  \\
  \hline

$s^\mu \d^{n-1}x_\mu$ & $0$ & $y^j \d^{n-1}x_\nu$ & $0$ & $\d^{n-1} x_\mu$ \\
\hline
$y^i \d^{n-1}x_\mu$ & $-y^i \d^{n-1}x_\mu$ & $0$ & $\delta^i_j \d^{n-1}x_\nu$ & $0$ \\
\hline
$p^\mu_i \d^{n-1}x_{\mu}$ & $0$ & $- \delta^j_i \d^{n-1}x_\mu$ & $0$ & $0$ \\
\hline
$\d^{n-1}x_\mu$ & $-\d^{n-1}x_\mu$ & $0$ & $0$ & $0$
\end{tabular}
\label{table:brackets}
\end{table}

Now, our goal is to write the evolution of the previous Hamiltonian forms with a fixed Hamiltonian (and hence fixed dynamics) using the studied geometry. Intrinsically, a Hamiltonian is given by a \textit{Hamiltonian section}
\[h\colon \bigwedge^{n}_2Y \Big/ \bigwedge^n_1 Y \oplus \bigwedge^{n-1} X \longrightarrow M = \bigwedge^{n}_2Y \oplus \bigwedge^{n-1} X\,,\]
which then induces a multicontact structure on $\bigwedge^{n}_2Y / \bigwedge^n_1 Y \oplus \bigwedge^{n-1} X$ using $\Theta_h = h^\ast \Theta$. As stated in the previous section, we prefer to work with the canonical multicontact structure on $M$ and think of the Hamiltonian as $\widetilde h = \tau^\ast \Theta_h - \Theta$, where $\tau\colon M \rightarrow \bigwedge^{n}_2Y / \bigwedge^n_1 Y \oplus \bigwedge^{n-1} X$ denotes the projection, hence fixing the geometry. The canonical coordinates on $\bigwedge^{n}_2Y / \bigwedge^n_1 Y \oplus \bigwedge^{n-1} X$ are $(x^\mu, y^i, p^\mu_i, s^\mu)$ and, hence, the local expression of the Hamiltonian section will be $p = - H(x^\mu, y^i, p^\mu_i, s^\mu)$, so that \[\Theta_h = \d s^\mu \wedge \d^{n-1} x_\mu + H \d^n x - p^\mu_i \d y^i \wedge \d^{n-1}x_\mu\,,\]
and 
\[\widetilde h = (p + H) \d^n x\,.\]
This implies that $\widetilde h$ is a good Hamiltonian, in the sense of Section \ref{section:Field_equations}, so that it could define the following dynamics for some $\psi \colon X \longrightarrow 
 M$:
\begin{equation*}
    \psi^{\ast} (\Theta + h) = 0\qquad\text{and}\qquad 
    \psi^{\ast} \iota_{\xi} (\d + \sigma_h \wedge)(\Theta + h) =  0\,,\quad \forall \xi \in \X(M)\,,
\end{equation*}
where $\sigma_h$ is the associated dissipation $1$-form, which takes the expression $\sigma_h = \pdv{H}{s^\mu} \d x^\mu$. To obtain the local expression of the Hamilton--de Donder--Weyl equations we will use Corollary \ref{cor:Evolutio_of_vertical_forms} and the conformal Hamiltonian forms of the previous table. Employing Theorem \ref{thm:Evolution_of_Hamiltonian_forms}, a direct computation yields:
\begin{align*}
    \psi^\ast(\d s^\mu \wedge \d^{n-1}x_\mu) &= \psi^\ast \left( \left(p^\mu_i \pdv{H}{p^\mu_i} - H\right) \d^n x \right)\,,\\
     \psi^\ast(\d y^i \wedge \d^{n-1}x_\mu) &= \psi^\ast \left( \pdv{H}{p^\mu_i} \d^n x \right)\,,\\
     \psi^\ast(\d p^\mu_i \wedge \d^{n-1}x_\mu) &= \psi^\ast \left( -\left(\pdv{H}{y^i} + \pdv{H}{s^\mu} p^\mu_i\right) \d^n x \right)\,.
\end{align*}
In particular, if we look for sections $\psi\colon X \rightarrow M$, we obtain the so-called multicontact Hamilton--de Donder--Weyl equations for dissipated field theories \cite{LGMRR_23,GGMRR_20,GLMR_24}:
\[\pdv{s^\mu}{x^\mu} = p^\mu_i \pdv{H}{p^\mu_i} \,,\qquad \pdv{y^i}{x^\mu} = \pdv{H}{p^\mu_i}\,, \qquad \pdv{p^\mu_i}{x^\mu} =  -\left(\pdv{H}{y^i} + \pdv{H}{s^\mu} p^\mu_i\right)\,.\]


\section{Conclusions and further work}

In this article we defined and studied brackets induced by an arbitrary $n$-form, as well as the relation of these brackets with the usual bracket appearing in multisymplectic geometry through the process of multisymplectization (see Theorem \ref{thm:multisymplectization_and_bracket} and Theorem \ref{thm:L_infinity}). Finally, we computed the Jacobi bracket of some conformal Hamiltonian forms present in the multicontact formulation of action dependent field theories. The definition and properties of these bracket open a few possible directions for research:
\begin{enumerate}[\rm (i)]
    \item With the introduction of the $\sharp$ morphism, one could do a preliminary classification of the submanifolds of a multicontact manifold, defining isotropic, coisotropic and Legendrian submanifolds (just as it is done in Jacobi manifolds is general). A study of these and some possible ``Graded Jacobi structures'' or ``Higher Jacobi structures'' would be of interest to understand the underlying structure of the brackets.
    \item Although we investigated a possible formulation of dynamics using the $\sharp$ and Reeb morphisms, an explicit description of the equations using Jacobi brackets would be of interest. This would allow for a description of reduction, both coisotropic and by symmetries.
    \item We also propose as a research topic the relation between the multicontact formulation of action dependent field theories and the contact infinite dimensional formulation using a space-time splitting. It should be the case that when integrated, the Jacobi bracket of $(n-1)$-forms recovers the usual Jacobi bracket of contact geometry.
    \item We want to compare the current definition of multicontact form  with previous multicontact definitions, and try to obtain a multicontact structure from a given $k$-contact \cite{GGMRR_20} or $k$-cocontact \cite{Riv_23} (see similar results for multisymplectic forms and $k$-cosymplectic structures \cite{LMORS_98}).
\end{enumerate}

\addcontentsline{toc}{section}{Acknowledgments}
\section*{Acknowledgments}

We acknowledge financial support of the 
{\sl Ministerio de Ciencia, Innovaci\'on y Universidades} (Spain), grants PID2021-125515NB-C21, PID2022-137909NB-C21, and RED2022-134301-T of AEI.
M. de Le\'on also acknowledges financial support from the Severo Ochoa Programme for Centers of Excellence in R\&D, CEX-2023-001347-S. We also wish to thank the anonymous referee for the very helpful comments, which have greatly improved the quality of this manuscript.

\section*{Conflict of interest}
On behalf of all authors, the corresponding author states that there is no conflict of interest.


\bibliographystyle{abbrv}
{\small
\bibliography{references.bib}
}

\end{document}